\documentclass[psamsfonts]{amsart}

\usepackage{dimadefs}
\usepackage{amsmath}
\usepackage{amsthm}
\usepackage{enumerate}
\usepackage{hyperref}
\usepackage{fullpage}
\usepackage{rotating}
\usepackage{pstricks}
\usepackage{graphicx}
\usepackage{tikz}
\usepackage{subfigure}

\numberwithin{equation}{section}

\DeclareMathOperator{\shift}{E}
\DeclareMathOperator{\id}{I}
\DeclareMathOperator{\Op}{\mathfrak{D}} 
\DeclareMathOperator{\diffop}{\mathcal{E}} 
\newcommand{\ggdiffop}[3]{\diffop_{#2,#3}^{#1}}
\newcommand{\gdiffop}[1][\theorder]{\ggdiffop{#1}{a}{b}}
\DeclareMathOperator{\dmop}{\ensuremath{\Theta}}
\DeclareMathOperator{\annop}{\ensuremath{\mathfrak{S}}} 

\DeclareMathOperator{\diag}{diag}
\DeclareMathOperator{\fdiff}{\Delta}
\DeclareMathOperator{\newop}{\ensuremath{\widehat{\Op}}}

\newcommand{\nullsp}[1]{\ensuremath{\mathcal{N}_{#1}}} 
\renewcommand{\vec}[1]{\ensuremath{\mathbf{#1}}} 
\newcommand{\ff}[2]{\ensuremath{(#1)_{#2}}} 
\newcommand{\bas}{\ensuremath{u}} 
\newcommand{\fn}{\ensuremath{g}} 
\newcommand{\n}[1]{\ensuremath{\widetilde{\fn_{#1}}}} 
\newcommand{\theorder}{\ensuremath{N}} 
\newcommand{\polymaxorder}{\ensuremath{k}} 
\newcommand{\np}{\ensuremath{\mathcal{K}}} 
\newcommand{\nders}{\ensuremath{S}} 
\newcommand{\coeff}{\ensuremath{p}} 
\newcommand{\polycoeff}{\ensuremath{a}} 
\newcommand{\newcoeff}{\ensuremath{\widehat{\coeff}}}
\newcommand{\nm}{\ensuremath{M}} 
\newcommand{\nmH}{\ensuremath{\widehat \nm}}
\newcommand{\nmC}{\ensuremath{\widetilde \nm}}

\newcommand{\gseq}[4]{\ensuremath{{#4}_{#3}^{(#1,#2)}}} 
\newcommand{\mm}[3]{\gseq{#1}{#2}{#3}{v}} 
\newcommand{\mmm}{\mm{i}{j}{k}}
\newcommand{\step}{\ensuremath{\mathcal{H}}}
\newcommand{\badpart}{\ensuremath{\varepsilon}}
\newcommand{\goodpart}{\ensuremath{\mu}}
\newcommand{\gprc}[2]{\gseq{#1}{#2}{}{\Pi}}
\newcommand{\prc}{\gprc{i}{j}} 

\newcommand{\pars}{\ensuremath{\mathcal{P}}}
\newcommand{\inp}{\ensuremath{\mathcal{I}}}
\newcommand{\fwm}{\ensuremath{\mathcal{M}}} 
\newcommand{\ivm}{\ensuremath{\mathcal{N}}} 

\newcommand{\er}{\ensuremath{\epsilon}}

\newcommand{\paclass}{\ensuremath{A_{\Op}}} 
\newcommand{\rpaclass}{\paclass^*} 
\newcommand{\stepdefinition}{\ensuremath{\begin{cases}0 & x<0\\ 1 & x \geq 0\end{cases}}}

\theoremstyle{plain}
\newtheorem{thm}{Theorem}[section]
\newtheorem*{uthm}{Theorem} 
\newtheorem{lem}[thm]{Lemma}

\newtheorem{prop}[thm]{Proposition}
\newtheorem*{uprop}{Proposition} 

\theoremstyle{definition}
\newtheorem{defn}[thm]{Definition}
\newtheorem*{RP}{Operator-Based Moment Reconstruction Problem}
\newtheorem{example}{Example}[section]

\theoremstyle{remark}
\newtheorem{rem}{Remark}[section]

\begin{document}

\title{Moment inversion problem for piecewise $D$-finite functions}
\author{Dmitry Batenkov}
\address{Department of Mathematics, Weizmann Institute of Science\\Rehovot 76100, Israel}
\email{dima.batenkov@weizmann.ac.il}
\date{\today}

\begin{abstract}
We consider the problem of exact reconstruction of univariate functions with jump discontinuities at unknown positions from their moments. These functions are assumed to satisfy an \emph{a priori unknown} linear homogeneous differential equation with polynomial coefficients on each continuity interval. Therefore, they may be specified by a finite amount of information. This reconstruction problem  has practical importance in Signal Processing and other applications.

 It is somewhat of a ``folklore'' that the sequence of the moments of such ``piecewise $D$-finite''functions  satisfies a linear recurrence relation of bounded order and degree. We derive this recurrence relation explicitly. It turns out that the coefficients of the differential operator which annihilates every piece of the function, as well as  the locations of the discontinuities, appear in this recurrence in a precisely controlled manner. This leads to the formulation of a generic algorithm for reconstructing a  piecewise $D$-finite function from its moments.
We investigate the conditions for solvability of the resulting linear systems in the general case, as well as analyze a few particular examples. We provide results of numerical simulations for several types of signals, which test the sensitivity of the proposed algorithm to noise.
\end{abstract}
\subjclass[2000]{Primary: 44A60; Secondary: 34A55, 34A37}
\maketitle

\section{Introduction}\label{sec:intro}

Consider the problem of reconstructing an unknown function $\fn:[a,b] \to \reals$ from some finite number of its power moments
\begin{align}\label{eq:usual-moments}
m_k(\fn)=\int_a^b x^k \fn(x) dx \quad k=0,1,\dots,\nm
\end{align}

This formulation is a ``prototype'' for various problems in Signal Processing, Statistics, Computer Tomography and other areas (see \cite{ang2002mta,golub2000snm,sig_ack} and references there). In all practical applications, it is assumed that $\fn$ belongs to some \emph{a priori known} class, and it can be faithfully specified by a \emph{finite} number of parameters in that class. For example, smooth signals may be represented as elements of some finite-dimensional Hilbert space and then the reconstruction problem is analyzed in the classical framework of linear approximation -- see \cite{talenti1987rff,ang2002mta} for thorough expositions.

In recent years, novel algebraic techniques for moment inversion are being developed. One reason for their appearance is the unsatisfactory performance of the classical approximation methods when applied to irregular data. The famous ``Gibbs effect'' due to a jump discontinuity probably provides the best-known example of such undesired behaviour. In \cite{ettinger2008lvn} it is shown that only nonlinear methods have a chance to achieve the same order of approximation for such discontinuous signals as conventional (linear) methods do for smooth signals. The various nonlinear methods developed recently include the framework of signals with finite rate of innovation (\cite{vetterli2002ssf,dragotti2007sma,berent2006prs}), Pad\'{e}-based methods (\cite{driscoll2001pba, march1998apm, marchbarone:2000}), and other algebraic schemes (\cite{eckhoff1993,kvernadze2004ajd,beckermann2008rgp}). Methods for reconstructing planar shapes from complex moments are developed in \cite{gustafsson2000rpd,golub2000snm}. In \cite{kisunko1} an approach based on Cauchy-type integrals is presented. In \cite{sig_ack}, several of the above methods are reviewed in detail, while emphasizing the mathematical similarity between the corresponding inversion systems.

The present paper develops a general method for explicit inversion of the moment transform for functions which are piecewise solutions of linear ODEs with polynomial coefficients (precise definitions follow). In particular, the method allows to locate the discontinuities of the function by purely algebraic means. The algebra involved is in the spirit of holonomic combinatorics (\cite{zeilberger1990hsa}). We believe that the tools developed in this paper may shed a new light on the similar structure of the algebraic equations which appear in \cite{vetterli2002ssf,eckhoff1993,march1998apm,gustafsson2000rpd,sig_ack}.

\subsection{Overview of main results}

Let $\Op$ denote an arbitrary linear differential operator with polynomial coefficients
\begin{subequations}\label{eq:operator-full-def}
\begin{align}
\label{eq:operator}  \Op &= \sum_{j=0}^\theorder \coeff_j(x) \partial^j\\
 \label{eq:thecoefficients} \coeff_j(x) &= \sum_{i=0}^{\polymaxorder_j} \polycoeff_{i,j}x^i \qquad \polycoeff_{i,j} \in \reals
\end{align}
\end{subequations}
where $\partial$ is the differentiation operator with respect to $x$: \f{\partial^i f \isdef \frac{d^i}{dx^i}f=\der{f}{n}(x)} and \f{\partial^0 = \id}, the identity operator. If some function $\fn$ satisfies \f{\Op \fn \equiv 0}, we say that $\Op$ \emph{annihilates} $\fn$ and also write \f{\fn \in \nullsp{\Op} \isdef \{f: \Op f \equiv 0\}}. These functions are called ``$D$-finite'' (``differentiably-finite'', \cite{stanley1980dfp}).

The class $\paclass$ of ``piecewise $D$-finite'' functions is defined in the following way: for \f{n = 0,1,\dots,\np}, let \f{\Delta_n \isdef [\xi_{n},\xi_{n+1}]} be a partition of \f{[a,b]} such that \f{-\infty < a=\xi_0 < \xi_1 < \dotsc < \xi_{\np+1}=b < +\infty} (the case $\np=0$ corresponds to functions consisting of a single ``piece''). We say that \f{\fn \in \paclass} if there exist operators $\Op_n$ such that on each ``continuity interval'' $\Delta_n$, the $n$-th piece of $\fn$ equals to, say, $\fn_n(x)$, such that \f{\Op_n \fn_n \equiv 0}. In the most general setting, the annihilating operators $\Op_n $ may be pairwise different. In this paper we explicitly treat functions for which \f{\Op_n \equiv \Op} for all \f{0 \leq n \leq \np}. We write \f{\fn \in \rpaclass} in this particular case.

We shall always assume that the leading coefficient of \f{\Op_n} does not vanish at any point of \f{\Delta_n}.

We are interested in solving the following
\begin{RP}
Given the sequence of the moments \eqref{eq:usual-moments} of an unknown function \f{\fn \in \paclass} and the constants \f{\np,\theorder,\{\polymaxorder_j\},a,b}:
\begin{enumerate}[1)]
 \item\label{step:operator} Find the operator $\Op$ (in the general case, the operators $\Op_n$) - that is, determine the coefficients $\polycoeff_{i,j}$ (respectively $\polycoeff_{i,j,n}$) as in \eqref{eq:operator-full-def} such that $\Op_n \fn_n \equiv 0$ for all \f{0 \leq n \leq \np}.
\item\label{step:jumps} Find the ``jump points'' $\{\xi_n\}_{n=1}^\np$.
\item Let \f{\{\bas_i\}_{i=1}^\theorder} be a basis for the linear space \f{\nullsp{\Op}}. Find constants \f{\alpha_{i,n}} such that \f{\fn_n(x)=\sum_{i=1}^\theorder \alpha_{i,n} \bas_i(x)}. That is, determine each concrete solution of \f{\Op_n \fn_n = 0} on $\Delta_n$.
\end{enumerate}
\end{RP}

\begin{example}\label{ex:direct-exponential}
Let \f{\fn(x)=\alpha e^{\beta x}} be our unknown function on \f{[0,1]}. The parameters of the problem are: \f{\np=0,\Op=\partial - \beta \id, [a,b]=[0,1]} and the unknowns are: \f{\alpha,\beta}. Direct calculation of the moments yields
\begin{align}\label{eq:constsystem1}
 \left\{ \begin{aligned}
          \alpha(e^\beta-1) &= \beta m_0\\
          \alpha e^\beta &= \beta m_1+m_0
         \end{aligned}\right.
\end{align}
Denoting $\frac{m_1}{m_0} \isdef A$, \eqref{eq:constsystem1} amounts to:
\begin{align}\label{eq:constsystem1sol1}
 \left\{ \begin{aligned}
   \frac{1}{1-e^{-\beta}} - \frac{1}{\beta} &= A\\
  \frac{\beta m_0}{e^\beta-1} &= \alpha
 \end{aligned}\right.
\end{align}
This system does not have an explicit analytic solution.  It can be solved numerically (for example, using Newton's method) and a unique solution exists provided \f{A \in (0,1)} (the graph of the function \f{y=\frac{1}{1-e^{-x}} - \frac{1}{x}} is monotone in whole $\reals$ and \f{0 < y(x) < 1}). \qed
\end{example}

The above solution is unsatisfactory for several reasons. First, it is not general enough. Second, the solution is available only as an approximation and not in a closed form. However, there is an important positive feature: the minimal possible number of measurements is used. Using our method, Example \ref{ex:direct-exponential} will be solved later in a more convenient and general way -- see Examples  \ref{ex:exponential-recurrence} and \ref{ex:firstorder}.

Our method is based on the following results which we prove below (Sections \ref{sec:single-interval}, \ref{sec:piecewise-dfinite} and \ref{sec:forward-equations}). These results establish explicit relations between the known and the unknown parameters of the reconstruction problem.

\begin{uthm}[\ref{thm:main-recurrence-theorem}, \ref{thm:operator-coeffs-main-thm}]
 Let \f{\np=0} and \f{\Op \fn \equiv 0}. Then the moment sequence of $\fn$ satisfies a linear recurrence relation with coefficients \emph{linear} in $\polycoeff_{i,j}$. Consequently, the vector \f{\vec{\polycoeff}=(\polycoeff_{i,j})} satisfies a linear homogeneous system \f{H\vec{a}=0} where the entries of $H$ are certain linear combinations of the moments of $\fn$ whose coefficients depend only on \f{\theorder} and the endpoints \f{a,b}.
\end{uthm}

\newcommand{\enlargedop}{\ensuremath{\bigr(\prod_{i=1}^\np (x-\xi_i)^\theorder \id \bigl) \cdot \Op}}
\begin{uthm}[\ref{thm:piecewise-annihilating-operator}]
 Let \f{\np > 0} and let $\Op$ annihilate every piece of \f{\fn \in \rpaclass}. Then the operator given by \f{\newop=\enlargedop} annihilates $\fn$ \emph{as a distribution}.  Consequently, conclusions of Theorems \ref{thm:main-recurrence-theorem} and \ref{thm:operator-coeffs-main-thm} are true with $\Op$ replaced by $\newop$.
\end{uthm}

\begin{uprop}[\ref{prop:particular-solution}]
 Let \f{\fn \in \rpaclass} with operator \f{\Op} annihilating every piece \f{\fn_n}. Let \f{\{\bas_i\}_{i=1}^\theorder} be a basis for the linear space \f{\nullsp{\Op}} and \f{\fn_n(x)=\sum_{i=1}^\theorder \alpha_{i,n} \bas_i(x)}. Then the vector \f{\alpha} of the coefficients \f{\alpha_{i,n}} satisfies a linear system \f{C\alpha = \vec{m}} where the matrix $C$ contains the moments of \f{\bas_i} and the vector \f{\vec{m}} contains the moments of \f{\fn}.
\end{uprop}

Based on the above results, the proposed solution to the reconstruction problem is as follows (Section \ref{sec:thealg}):
\begin{enumerate}
 \item If \f{\np > 0}, replace \f{\Op} with \f{\newop=\enlargedop}.
\item Build the matrix $H$ and solve the system \f{H\vec{x}=0} where \f{\vec{x}} is the vector of unknown coefficients of $\Op$ or $\newop$ according to the previous step. Obtain a solution $\vec{a}$ and build the differential operator \f{\Op^*=\Op_{\vec{a}}} which annihilates $\fn$ in its entirety (as a distribution in case \f{\np>0} or in the usual sense otherwise).
\item If \f{\np>0}, recover \f{\{\xi_i\}} and the operator \f{\Op^\dagger} (which annihilates every piece of $\fn$) from \f{\Op^*}.
\item Compute the basis for \f{\nullsp{\Op^\dagger}} and solve the system \f{C \alpha = \vec{m}}.
\end{enumerate}

Conditions for the solvability of the above systems are discussed and few initial results in this direction are obtained in Section \ref{sec:solvability} below. Results of numerical simulations, presented in Section \ref{sec:stability}, suggest that the method is practically applicable to several different signal models, including piecewise constant functions (and piecewise polynomials), rational functions and piecewise sinusoids.

\subsection{Acknowledgements}
The author wishes to thank Yosef Yomdin for useful discussions and remarks. Also, the criticism of the reviewers has been very helpful.

\section{Recurrence relation for moments of piecewise $D$-finite functions}\label{sec:recurrence}
\subsection{Single continuity interval}\label{sec:single-interval}
\newcommand{\ggconcom}[3]{\ensuremath{P_{#3}(#1,#2)}}
\newcommand{\gconcomlimits}[5]{\ensuremath{\left[ \ggconcom{#1}{#2}{#3} \right]_{#4}^{#5}}}
\newcommand{\gconcom}[2]{\ensuremath{\gconcomlimits{\fn}{x^k}{\Op}{#1}{#2}}} 
\newcommand{\concom}{\gconcom{a}{b}} 

We start with the case of a $D$-finite function $\fn$ on a single continuity interval \f{[a,b]}. The main tools used in the subsequent derivation are the discrete difference calculus and the Lagrange identity for a differential operator and its adjoint. Let us briefly introduce these tools.

Let \f{s: \naturals \to \reals} be a discrete sequence. The \emph{discrete shift operator} $\shift$ is defined by \f{\shift s(n) = s(n+1)}, and the forward difference operator by \f{\fdiff \isdef \shift - \id}. We shall express recurrence relations in terms of polynomials in $\shift$ or $\fdiff$. For example, the Fibonacci sequence $F_k$ satisfies \f{\shift^2F_k=\shift F_k+F_k}, so the operator \f{P(\shift) \isdef \shift^2-\shift-\id} is an annihilating difference operator for $F_k$. Likewise, the operator $\fdiff$ annihilates every constant sequence.

\begin{lem}[\cite{elaydi2005ide}]\label{lem:polynomial-shift}
Let $p(\shift)$ be a polynomial in the shift operator $\shift$ and let $g(n)$ be any discrete function. Then
\[
p(\shift)(b^n g(n)) = b^n p(b\shift) g(n)
\]
\end{lem}
\begin{lem}[\cite{elaydi2005ide}]\label{lem:discrete-polynomial-annihilated-delta-operator}
For any polynomial $p(n)$ of degree $k$ and for $i \geq 1$:
\[\fdiff^{k+i} p(n) = 0\]
\end{lem}
\begin{defn}\label{def:ff}
 Let \f{x \in \reals} and \f{k \in \naturals}. The $k$th \emph{falling factorial}\footnote{ The Pochhammer symbol \ff{x}{n} is also used in the theory of special functions. There it usually represents the rising factorial \f{\ff{x}{n}=x\cdot(x+1)\cdot\dotsc \cdot(x+n-1)}. } of $x$ is
\[
 \ff{x}{k} \isdef x(x-1)\dots(x-k+1)
\]
\end{defn}

The following well-known properties of the falling factorial are immediately derived from Definition \ref{def:ff}:
\begin{prop}
Let \f{x \in \reals}, \f{k \in \naturals}. Then
\begin{enumerate}[(1)]
\item \f{\ff{x}{k}} is a polynomial in $x$ of degree $k$ (thus it is also called the \emph{factorial polynomial}).
 \item If \f{x = n \in \naturals} and \f{n \geq k}, then \f{\ff{n}{k} = \frac{n!}{(n-k)!}}.
\item If \f{x = n \in \naturals \cup \{0\}} and \f{n < k}, then \f{\ff{n}{k}=0}.
\end{enumerate}
\end{prop}

The formal adjoint of a differential operator \f{\Op=\sum_{j=0}^\theorder \coeff_j(x) \partial^j} is given by
\begin{equation}\label{eq:adjoint-definition}
 \Op^*\{\cdot\} \isdef \sum_{j=0}^\theorder (-1)^j \partial^j \{\coeff_j(x) \cdot \}
\end{equation}
The operator and its adjoint are connected by the \emph{Lagrange identity} (\cite{ince1956ode}): for every \f{u,v \in C^{\theorder}}
\begin{equation}\label{eq:lagrange}
 v \Op(u) - u \Op^*(v) = \frac{d}{dx} \ggconcom{u}{v}{\Op}
\end{equation}
where $\ggconcom{u}{v}{\Op}$ is the \emph{bilinear concomitant} - a homogeneous bilinear form which may be written explicitly as (\cite[p.211]{ince1956ode}):
\begin{align}\label{eq:concomitant-explicit}
\begin{split}
\ggconcom{u}{v}{\Op} &= u \left\lbrace \coeff_1 v - \partial(\coeff_2 v) + \cdots + (-1)^{\theorder-1} \partial^{\theorder-1} (\coeff_{\theorder} v) \right\rbrace\\
&+ u' \left\lbrace \coeff_2 v - \partial (\coeff_3 v) + \cdots + (-1)^{\theorder-2} \partial^{\theorder-2} (\coeff_{\theorder} v) \right\rbrace\\
&+ \cdots\\
&+ \der{u}{\theorder-1} \coeff_{\theorder} v
\end{split}
\end{align}

If \eqref{eq:lagrange} is integrated between $a$ and $b$, \emph{Green's formula} is obtained:
\begin{equation}\label{eq:green}
 \langle \Op u,v \rangle - \langle u,\Op^* v \rangle = \gconcomlimits{u}{v}{\Op}{a}{b}
\end{equation}
where the inner product is defined by\f{\langle u,v \rangle \isdef \int_a^b u v dx}.

Let $\Op$ and $\fn$ be arbitrary. Consider the ``differential moments'' associated with $\Op$: \[
m_k^{\Op}(\fn) \isdef m_k(\Op \fn)
\]
By \eqref{eq:green}, we have
\[
m_k(\Op \fn)=\langle \Op \fn, x^k \rangle = \langle \fn, \Op^*(x^k) \rangle + \concom
\]
Let us define the following two sequences, indexed by $k$:
\begin{align*}
\goodpart_k&=\goodpart_k(\Op,g) \isdef \langle \fn, \Op^*(x^k) \rangle\\
\badpart_k&=\badpart_k(\Op,\fn) \isdef \concom
\end{align*}
The sequence of the differential moments is therefore the sum of the two sequences:
\begin{align}\label{eq:diffmoments-defs}
m_k(\Op \fn) &= \goodpart_k(\Op,\fn) + \badpart_k(\Op,\fn)
\end{align}

Using \eqref{eq:operator-full-def} and \eqref{eq:adjoint-definition} we have:
\begin{align}\label{eq:diffmoments-linear-term}
 \begin{split}
  \goodpart_k = \langle \fn, \Op^*(x^k) \rangle &= \langle \fn, \sum_{j=0}^\theorder (-1)^j \frac{d^j}{dx^j}(\coeff_j(x) x^k) \rangle\\ &=  \sum_{i=0}^{\polymaxorder_j} \sum_{j=0}^{\theorder} \polycoeff_{i,j} (-1)^j \ff{i+k}{j} m_{i+k-j}(\fn) \isdef \dmop_{\Op}(k,\shift)m_k
 \end{split}
\end{align}
where
\begin{align}\label{eq:mmm-def}
\dmop=\dmop_{\Op}(k,\shift) &\isdef \sum_{i,j} \polycoeff_{i,j} \prc(k,
\shift) & 
\prc(k,\shift) &\isdef (-1)^j \ff{i+k}{j}  \shift^{i-j}
\end{align}

Now let us return to our main problem. Recall that $\fn$ is $D$-finite, so let \f{\Op \fn = 0}. \eqref{eq:diffmoments-defs} combined with \eqref{eq:diffmoments-linear-term} gives \f{\dmop m_k + \badpart_k = 0}. As we demonstrate below, there exists a discrete difference operator \f{\diffop=\diffop(\shift)} such that $\diffop \badpart_k \equiv 0$. Multiplying the last equation  by this $\diffop$ from the left (multiplication being composition of difference operators) gives us the desired recurrence relation: \f{\diffop \cdot \dmop m_k = 0}.

The sequence $\badpart_k$ is related to the behavior of $\fn$ at the endpoints of the interval \f{[a,b]}. The following lemma unravels its structure.

\begin{lem}\label{lem:concomitant-structure}
Let \f{\Op} be of degree $\theorder$ as in \eqref{eq:operator-full-def}. Then there exist polynomials $q_a(k)$ and $q_b(k)$ of degree at most $\theorder-1$ such that
 \begin{align}
  \badpart_k (\Op,\fn) &= b^k q_b(k) - a^k q_a(k), \quad k=0,1,\dotsc
 \end{align}
\begin{proof}
 Write \f{\Op = \sum_{j=0}^\theorder \Op_j} where \f{\Op_j=\coeff_j(x) \partial^j}. Denote \f{\badpart_{k,j} \isdef \badpart_k(\Op_j,\fn)}. By \eqref{eq:concomitant-explicit} we have \f{\badpart_k = \sum_{j=0}^\theorder \badpart_{k,j}} where
\begin{align*}
\badpart_{k,j} &= \ggconcom{\fn}{x^k}{\Op_j}= \bigg\lbrace \der{\fn}{j-1} x^k p_j - \der{\fn}{j-2} \partial (x^k p_j) + \dots + (-1)^{j-1} \fn \partial^{j-1}(x^k p_j) \bigg\rbrace \bigg|_{x=a}^b
\end{align*}
Use Leibniz rule:
\begin{align*}
\partial^i (x^k p_j(x)) &= \sum_{l=0}^{i} {i \choose l} \der{(x^k)}{l} \der{p_j}{i-l}(x) = \sum_{l=0}^{i} {i \choose l} \der{p_j}{i-l}(x) \ff{k}{l} x^{k-l} = x^k \sum_{l=0}^{i} \ff{k}{l} r_{i,j,l}(x)
\end{align*}
where \f{r_{i,j,l}(x) = x^{-l} {i \choose l} \der{p_j}{i-l}(x)} is a rational function. Now
\begin{align*}
\badpart_{k,j} &= \left. \left \lbrace x^k \sum_{i=0}^{j-1} (-1)^i \der{\fn}{j-1-i}(x) \sum_{l=0}^{i} \ff{k}{l} r_{i,j,l}(x) \right\rbrace \right|_{x=a}^b= \left. \left \lbrace x^k \sum_{i=0}^{j-1} \sum_{l=0}^{i} \ff{k}{l} s_{i,j,l}(x) \right\rbrace \right|_{x=a}^b\\ &= b^k q_{b,j}(k) - a^k q_{a,j}(k)
\end{align*}
where \f{s_{i,j,l}(x) = (-1)^i \der{\fn}{j-1-i}(x) r_{i,j,l}(x)} and \f{
q_{\alpha,j}(k) = \sum_{i=0}^{j-1} \sum_{l=0}^{i} \ff{k}{l} s_{i,j,l}(\alpha)} for \f{\alpha \in \{a,b\}}.
These are polynomials in $k$ of degree at most \f{j-1} (see Definition \ref{def:ff}). Now
\begin{align*}
 \badpart_k &= \sum_{j=0}^\theorder \badpart_{k,j} = b^k \sum_{j=0}^\theorder q_{b,j}(k) - a^k \sum_{j=0}^\theorder q_{a,j}(k)
\end{align*}
Take \f{q_a(k) \isdef \sum_{j=0}^\theorder q_{a,j}(k)} and \f{q_b(k) \isdef \sum_{j=0}^\theorder q_{b,j}(k)}. Since \f{\deg q_{a,j},q_{b,j} < j}, then we have \f{\deg q_a,q_b < \theorder} and this completes the proof.
\end{proof}
\end{lem}

As a side remark, we have the following simple condition for the sequence \f{\{\badpart_k\}} to be a nonzero sequence.
\begin{thm}\label{thm:vanishing-concomitant}
Assume \f{\Op \fn \equiv 0} and \f{\coeff_\theorder(x) \neq 0} on \f{[a,b]}. Then \f{\badpart_k(\Op,\fn) \equiv 0} if and only if \f{\fn \equiv 0}.
\begin{proof}
From the proof of Lemma \ref{lem:concomitant-structure} we have for \f{\alpha \in \{a,b\}}:
\begin{align}\label{eq:concomitant-part-explicit}
 q_{\alpha}(k) &= \sum_{j=0}^\theorder \sum_{i=0}^{j-1} (-1)^i \der{\fn}{j-1-i}(\alpha) \sum_{l=0}^i \ff{k}{l} \alpha^{-l} {i \choose l} \der{\coeff_j}{i-l}(\alpha)
\end{align}

\begin{itemize}
 \item In one direction, we have \f{\fn(\alpha)=\dots=\der{\fn}{\theorder-1}(\alpha)=0} for \f{\alpha=a,b}. By direct substitution we obtain \f{q_{\alpha}(k) \equiv 0}.
\item To prove the other direction, assume \f{b^k q_b(k) \equiv a^k q_a(k) = C}. Consider the following cases:
\begin{enumerate}[(1)]
 \item \f{C \neq 0} i.e. \f{a,b \neq 0} and \f{q_a(k),q_b(k) \not\equiv 0}. Then \f{\Bigl(\frac{b}{a}\Bigr)^k = \frac{q_a(k)}{q_b(k)}} for all \f{k \in \naturals}. But this is impossible because the left hand side is of exponential growth and the right-hand side is of at most polynomial growth.
 \item \f{C = 0}. Then at least one of \f{q_a(k),q_b(k)} must be identically zero. So let \f{q_a(k) \equiv 0} and \f{a \neq 0}. The coefficient of the highest order term of \f{q_a} corresponds to \f{j=\theorder,i=l=\theorder-1} and equals to
\[
 (-1)^{\theorder-1} \fn(a) a^{-(\theorder-1)} \coeff_{\theorder}(a)
\]
This is zero only if \f{\fn(a)=0}. So we can lower the limit of the second summation in \eqref{eq:concomitant-part-explicit} to \f{n-2}. Repeating this argument with \f{j=\theorder,i=l=\theorder-2} shows that \f{\fn'(a)=0}. Finally we obtain that \f{\fn(a)=\dots=\der{\fn}{\theorder-1}(a)=0} and by the uniqueness theorem for linear ODEs we conclude \f{\fn \equiv 0}. \qedhere
\end{enumerate}
\end{itemize}
\end{proof}
\end{thm}

An immediate corollary is that generically we have \f{\diffop \neq \id} (\emph{generically} here means \f{\fn \not \equiv 0}).

With the structure of the sequence \f{\{\badpart_k\}} at hand, we can now explicitly construct an annihilating difference operator for it.

\begin{defn}\label{def:diffop}
Given $j$, $a$ and $b$, let \f{\gdiffop[j](\shift) \isdef (\shift-a\id)^{j} (\shift-b\id)^{j}}.
\end{defn}

\begin{thm}\label{thm:main-differences-generalized}
Let \f{\Op=\sum_{j=0}^\theorder \coeff_j(x) \partial^j}. Then the difference operator \f{\diffop=\gdiffop} annihilates the sequence \f{\badpart_k(\Op,\fn)}.
 \begin{proof}
By Lemma \ref{lem:concomitant-structure}, \f{\badpart_{k} = b^k q_{b}(k) - a^k q_{a}(k)} where \f{\deg q_{a},\;q_{b} \leq \theorder-1}. The factors $(\shift-a\id)$ and $(\shift-b\id)$ commute, so it would be sufficient to show that \f{(\shift-a\id)^\theorder \{a^k q_{a}(k)\} \equiv 0}.

Let \f{P(\shift)=(\shift-a\id)^{\theorder}}, then \f{P(a\shift) = a^{\theorder} \fdiff^{\theorder}}. By Lemma \ref{lem:polynomial-shift}, \f{P(\shift) \{a^k q_{a}(k)\} = a^{k+\theorder} \fdiff^{\theorder} q_{a}(k)}. Since \f{\deg q_{a} \leq \theorder-1}, by Lemma \ref{lem:discrete-polynomial-annihilated-delta-operator} we have \f{\fdiff^{\theorder} q_{a}(k) = 0}.
 \end{proof}
\end{thm}

\begin{rem}\label{rem:mgf-connection}
Theorem \ref{thm:main-differences-generalized} defines a connection between the moments of the two functions $\fn$ and \f{\Op \fn}: \f{\diffop m_k(\Op \fn) = \diffop \dmop m_k(\fn)}. This implies a connection between their moment generating functions as well (see further Section \ref{sec:forward-equations} and also \cite{kisunko1}).
\end{rem}

\begin{rem}
The polynomial nature of the coefficients $\coeff_j(x)$ of $\Op$ has not been used in the proof of Theorem \ref{thm:main-differences-generalized}. Therefore it is true for every linear operator $\Op$ with sufficiently smooth coefficients $\coeff_j(x)$.
\end{rem}

Now we have all the necessary information in order to prove the main result of this section.
\begin{thm}\label{thm:main-recurrence-theorem}
Assume that \f{\Op \fn \equiv 0} on \f{[a,b]}. Then the sequence $\{m_k(\fn)\}$ satisfies the following recurrence relation:
\begin{align}\label{eq:main-recurrence-explicit}
 \annop m_k \isdef \left(\gdiffop \cdot \dmop_{\Op}\right) m_k = \biggl((\shift-a\id)^\theorder (\shift-b\id)^\theorder \sum_{j=0}^\theorder \sum_{i=0}^{\polymaxorder_j} \polycoeff_{i,j} (-1)^j \ff{i+k}{j} \shift^{i-j}\biggr) m_k = 0
\end{align}
 \begin{proof}
 We have \f{m_k(\Op \fn) \equiv 0}. The proof is completed using \eqref{eq:diffmoments-defs}, \eqref{eq:diffmoments-linear-term} and Theorem \ref{thm:main-differences-generalized}.
 \end{proof}
\end{thm}

\begin{rem}
With respect to $\{m_k\}$, the length of the recurrence relation \eqref{eq:main-recurrence-explicit} is at most \f{3\theorder+\max \polymaxorder_j + 1}. Its coefficients are linear in \f{\polycoeff_{i,j}} and polynomial in $k$.
\end{rem}

\begin{rem}
 The recurrence relation \eqref{eq:main-recurrence-explicit} is not trivial, i.e. \f{\annop \neq 0}. To see this, let \f{\Op \fn \equiv 0}. Now if \f{\fn \not\equiv 0} then by Theorem \ref{thm:vanishing-concomitant} \f{\badpart_k \not\equiv 0} and thus \f{\goodpart_k=\dmop m_k \not\equiv 0 }. It follows that \f{\dmop} cannot be the identical zero operator, and therefore \f{\annop = \diffop \dmop \neq 0}.
\end{rem}

We conclude this section with some examples which demonstrate the usefulness of Theorem \ref{thm:main-recurrence-theorem}.
\begin{example}\label{ex:bessel}
In \cite{BoSa08}, the authors provide explicit recurrence relations satisfied by the moments of the powers of the \emph{modified Bessel function} \f{f(x)=K_0(x)}. The method used to obtain these recurrences is integration by parts. However, an additional condition is imposed - namely, it is required that the integrals \f{\int_{\Gamma} x^k \der{(f(x)^n)}{j} dx} converge and the limits of the integrands coincide at the endpoints of $\Gamma$. In our setting, this is equivalent to putting \f{\diffop=\id}.

The function \f{\fn=K_0^2(x)} is annihilated by the operator \f{\Op=x^2\partial^3+3x\partial^2+(1-4x^2)\partial-4x\id} (\cite[Example 2]{BoSa08}). The only nonzero coefficients are therefore
\[
\begin{tabular}{c|c|c|c}
$\polycoeff_{i,j}$ & $i=0$ & $i=1$ & $i=2$ \\
\hline
$j=0$ &  & -4 &  \\
$j=1$ & 1 &  & -4 \\
$j=2$ &  & 3 &  \\
$j=3$ &  &  & 1 \\
\end{tabular}
\]
By \eqref{eq:main-recurrence-explicit} we have
\[
-4m_{k+1}-km_{k-1}+4(k+2)m_{k+1}+3k(k+1)m_{k-1}-(k+2)(k+1)km_{k-1}=0
\]
which is just \f{4(k+1)m_{k+1}=k^3m_{k-1}}. This result agrees with \cite[Example 3]{BoSa08}.
\qed
\end{example}

\begin{example}[Example \ref{ex:direct-exponential} continued]\label{ex:exponential-recurrence}
 \f{f=\alpha e^{\beta x}} is annihilated by the operator \f{\Op=\partial - \beta \id}. Thus \f{\polycoeff_{0,0}=-\beta} and \f{\polycoeff_{1,0}=1}. By \eqref{eq:main-recurrence-explicit} the sequence $\{m_k\}$ on \f{[0,1]} satisfies
\[
 \annop = \shift (\shift-\id)(-\beta\id-k\shift^{-1})m_k=-\beta m_{k+2} + (\beta-(k+2)) m_{k+1} + (k+1)m_k = 0
\]
This recurrence relation with polynomial coefficients may be solved explicitly using computer algebra tools (see \cite{weixlbaumer2001sde} for an overview of the existing algorithms). Using the Maxima computer algebra system (\cite{schelter2001mm}), the following explicit formula was obtained:
\begin{align*}
 m_k &= \alpha \int_0^1 x^k e^{\beta x} dx = \frac{(-1)^k k!}{\beta^k} \biggl\{ c_1 + c_2 \sum_{j=0}^{k-1} \frac{\beta^j (-1)^{j-1}}{(j+1)!} \biggr\}\\
c_1 &= m_0\\
c_2 &= \beta m_1+m_0
\end{align*}
This example is further continued in Example \ref{ex:firstorder}.
\qed
\end{example}

\subsection{Piecewise case}\label{sec:piecewise-dfinite}
\begin{figure}
\centering
\psset{xunit=.5pt,yunit=.5pt,runit=.5pt}
\begin{pspicture}(663.50164795,324.19845581)
{
\newrgbcolor{curcolor}{0 0 0}
\pscustom[linewidth=1.46649683,linecolor=curcolor,linestyle=dashed,dash=4.39949053 4.39949053]
{
\newpath
\moveto(268.253057,318.09554581)
\lineto(268.253057,35.87023581)
}
}
{
\newrgbcolor{curcolor}{0 0 0}
\pscustom[linewidth=1.46649683,linecolor=curcolor,linestyle=dashed,dash=4.39949053 4.39949053]
{
\newpath
\moveto(153.680877,34.96191581)
\lineto(153.680877,317.18722581)
}
}
{
\newrgbcolor{curcolor}{0 0 0}
\pscustom[linewidth=1.46649683,linecolor=curcolor,linestyle=dashed,dash=4.39949053 4.39949053]
{
\newpath
\moveto(17.360282,41.23989581)
\lineto(17.360282,323.46520781)
}
}
{
\newrgbcolor{curcolor}{0 0 0}
\pscustom[linewidth=1.08071089,linecolor=curcolor]
{
\newpath
\moveto(15.266137,37.40652581)
\lineto(267.002347,36.03931581)
}
}
{
\newrgbcolor{curcolor}{0 0 0}
\pscustom[linewidth=1.46649683,linecolor=curcolor,linestyle=dashed,dash=4.39949053 4.39949053]
{
\newpath
\moveto(660.627627,318.09554581)
\lineto(660.627627,35.87023581)
}
}
{
\newrgbcolor{curcolor}{0 0 0}
\pscustom[linewidth=1.46649683,linecolor=curcolor,linestyle=dashed,dash=4.39949053 4.39949053]
{
\newpath
\moveto(473.305157,34.96191581)
\lineto(473.305157,317.18722581)
}
}
{
\newrgbcolor{curcolor}{0 0 0}
\pscustom[linewidth=1.46649683,linecolor=curcolor,linestyle=dashed,dash=4.39949053 4.39949053]
{
\newpath
\moveto(360.837117,41.23989581)
\lineto(360.837117,323.46520781)
}
}
{
\newrgbcolor{curcolor}{0 0 0}
\pscustom[linewidth=1.15019548,linecolor=curcolor]
{
\newpath
\moveto(358.771217,35.35582581)
\lineto(662.926537,34.07405581)
}
}
{
\newrgbcolor{curcolor}{0 0 0}
\pscustom[linewidth=0.82122475,linecolor=curcolor,linestyle=dashed,dash=2.46367434 2.46367434]
{
\newpath
\moveto(359.517737,35.90957581)
\lineto(270.013357,35.90957581)
}
}
{
\newrgbcolor{curcolor}{0 0 0}
\pscustom[linewidth=1.46649683,linecolor=curcolor]
{
\newpath
\moveto(17.393624,131.30954581)
\curveto(17.393624,131.30954581)(77.820097,97.64875581)(112.803847,131.30954581)
\curveto(147.787607,164.97030581)(153.353197,155.35294581)(153.353197,155.35294581)
}
}
{
\newrgbcolor{curcolor}{0 0 0}
\pscustom[linewidth=1.46649683,linecolor=curcolor]
{
\newpath
\moveto(153.353197,199.83325581)
\curveto(153.353197,199.83325581)(161.304047,229.88751581)(178.000837,211.85495581)
\curveto(194.697627,193.82239581)(209.009167,130.10736581)(233.656817,162.56596581)
\curveto(258.304447,195.02457581)(267.845467,184.20503581)(267.845467,184.20503581)
}
}
{
\newrgbcolor{curcolor}{0 0 0}
\pscustom[linewidth=1.46649683,linecolor=curcolor]
{
\newpath
\moveto(360.870437,150.54425581)
\curveto(360.870437,150.54425581)(391.878767,155.35294581)(412.550987,169.77898581)
\curveto(433.223207,184.20503581)(450.715077,204.64192581)(461.846267,228.68534581)
\curveto(469.002037,244.31355581)(472.977457,267.15480581)(472.977457,267.15480581)
}
}
{
\newrgbcolor{curcolor}{0 0 0}
\pscustom[linewidth=1.46649683,linecolor=curcolor]
{
\newpath
\moveto(472.977457,297.20906581)
\curveto(554.076157,275.56999581)(509.551387,209.45061581)(571.568037,170.98116581)
\curveto(633.584677,132.51169581)(659.822497,193.82239581)(659.822497,193.82239581)
}
}
{
\newrgbcolor{curcolor}{0 0 0}
\pscustom[linewidth=1,linecolor=curcolor,linestyle=dashed,dash=2 2]
{
\newpath
\moveto(270.523117,184.14826581)
\curveto(347.523117,134.14826581)(359.523117,151.14826581)(359.523117,151.14826581)
}
}
\put(-10,0){$a=\xi_0$}
\put(150,0){$\xi_1$}
\put(260,0){$\xi_2$}
\put(295,0){$\dots$}
\put(335,0){$\xi_{\np-1}$}
\put(460,0){$\xi_{\np}$}
\put(630,0){$\xi_{\np+1}=b$}
\tiny
\put(50,80){$\Op_0 \fn_0 = 0$}
\put(60,130){$\fn_0(x)$}
\put(170,80){$\Op_1 \fn_1 = 0$}
\put(180,250){$\fn_1(x)$}
\put(365,80){$\Op_{\np-1} \fn_{\np-1} = 0$}
\put(380,200){$\fn_{\np-1}(x)$}
\put(530,80){$\Op_\np \fn_\np = 0$}
\put(560,210){$\fn_\np(x)$}
\end{pspicture}
\caption{Piecewise $D$-finite function}
\label{fig:piecewise}
\end{figure}
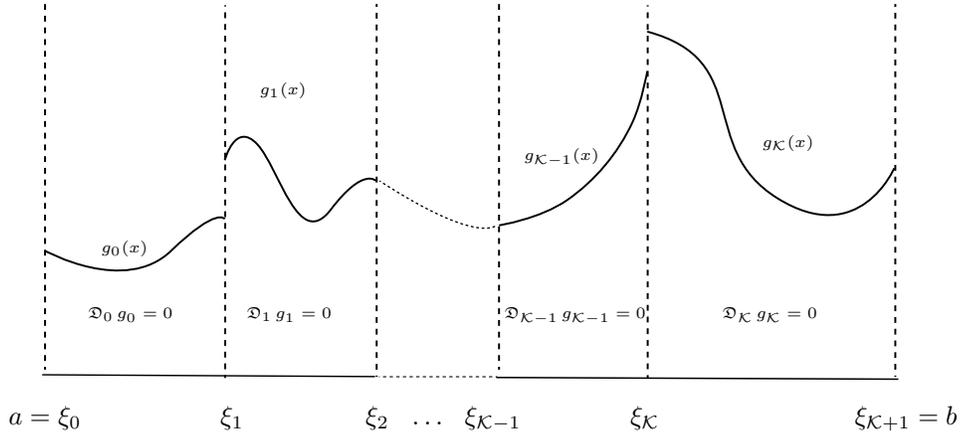

Until now we have been considering a function $\fn$ which satisfies \f{\Op \fn = 0} on a \emph{single interval} \f{[a,b]}. In particular, we have seen that the sequence of the moments of $\fn$ satisfies a linear  recurrence relation whose coefficients linearly depend on the coefficients of $\Op$. Now we are going to consider the piecewise case (depicted in Figure \ref{fig:piecewise}) where $\fn$ is assumed to consist of several ``pieces'' \f{\fn_0,\fn_1,\dotsc,\fn_\np}. On each continuity interval \f{\Delta_n=[\xi_n,\xi_{n+1}]} the $n$-th piece of $\fn$ satisfies \f{\Op_n \fn_n = 0} where \f{\Op_n = \sum_{j=0}^{\theorder_n} \sum_{i=0}^{\polymaxorder_n} (\polycoeff_{i,j,n} x^i) \partial^j}. Denote \f{m_{k,n} \isdef \int_{\xi_n}^{\xi_{n+1}} x^k \fn_n(x)dx}. Our goal is to find a recurrence relation satisfied by the sequence \f{m_k = \sum_{n=0}^\np m_{k,n}}.

As we shall see below, this recurrence relation is particularly easy to find explicitly in the case \f{\fn \in \rpaclass}, i.e. \f{\Op_n=\Op} for \f{n=0,\dotsc,\np}. We shall also discuss the general case where \f{\Op_m \neq \Op_n} briefly.

\subsubsection{The same operator on each interval}
We present two methods for computing the desired recurrence relation. Then we show that these  methods indeed produce the same result.

\begin{thm}[Method I]
Let \f{\fn \in \rpaclass} with \f{\Op_m=\Op} for \f{0 \leq m \leq \np}. Then $\{m_k(\fn)\}$ satisfies
\begin{align}\label{eq:recurrence-entire-piecewise}
 \biggl(\prod_{n=0}^{\np+1} (\shift-\xi_n \id)^\theorder \dmop_{\Op}(k,\shift) \biggr) m_k &= 0
\end{align}
\begin{proof}
 By Theorem \ref{thm:main-recurrence-theorem}, for every \f{n=0,\dots,\np} the sequence \f{\{m_{k,n}\}} satisfies
\begin{align}\label{eq:recurrence-for-each-piece}
\bigl(\ggdiffop{\theorder}{\xi_n}{\xi_{n+1}}(\shift) \cdot \dmop_{\Op}(k,\shift)\bigr) m_{k,n} &= 0
\end{align}
 \eqref{eq:recurrence-entire-piecewise} follows by using the fact that the linear factors \f{(\shift-\xi_i \id)} and \f{(\shift-\xi_j \id)} commute.
\end{proof}
\end{thm}

It turns out that the formula \eqref{eq:recurrence-entire-piecewise} may be obtained by considering the \emph{piecewise} function \f{\fn \in \rpaclass} being annihilated \emph{as a distribution on \f{[a,b]}} by some operator \f{\newop=\newop(\Op,\{\xi_n\})}. We now derive the explicit expression for $\newop$.

The general theory of distributions (generalized functions) may be found in \cite{gelfand1958gfa}. By a \emph{test function} we shall mean any \f{f \in C^{\theorder-1}([a,b])} (in fact, for our purposes it is sufficient to consider just the moments \f{f=x^k}).

We shall identify the discontinuous function $\fn$ with the following distribution (here \f{\fn_{-1} \equiv 0} by definition):
\begin{subequations}\label{eq:piecewise-full-def}
\begin{align}
\fn(x) &= \n{0}+\sum_{n=1}^\np \n{n}(x) \step(x-\xi_n)\\
\n{n} &\isdef \fn_n-\fn_{n-1}\\
\step(x) &\isdef \stepdefinition
\end{align}
\end{subequations}
The functions \f{\n{n}(x)} belong to the solution space of \f{\Op} on \f{\Delta_n}, and thus \f{\n{n}(x) \in C^{\theorder-1}(\Delta_n)}. Since \f{\coeff_n(x)\neq 0} on \f{[a,b]}, we even have \f{\n{n}(x) \in C^{\theorder-1}([a,b])}.

The derivative of $\step$ is the Dirac \f{\delta}. The distribution \f{\delta(x)} has the following properties:
\begin{align*}
 \langle \delta(x-t), f(x) \rangle &= f(t)\\
 \langle \der{\delta}{j}(x-t), f(x) \rangle &= (-1)^j \der{f}{j}(t)
\end{align*}
The second equality is valid provided $f(x)$ is $j$-times differentiable at \f{x=t}.

Now let \f{\fn(x) = c\cdot\delta(x-\xi)} for \f{c,\xi \in \reals}. This distribution is ``annihilated'' by the operator $\Op=(x-\xi)\id$ in the sense that \f{\langle \Op \fn, f \rangle = \langle  c \cdot (x-\xi)\delta(x-\xi),f \rangle = 0} for every test function $f$. By linearity, the operator \f{\Op=\prod_{i=1}^{\np}(x-\xi_i) \id} annihilates every $\fn$ of the form \f{\fn = \sum_{i=1}^\np c_i \delta(x-\xi_i)}.

\begin{lem}\label{lem:deltader-annihilation}
Let  \f{\fn = \sum_{i=1}^\np \sum_{j=0}^{\nders-1} c_{ij}(x) \der{\delta}{j}(x-\xi_i)} such that \f{c_{ij}(x)} is \f{j}-times differentiable at  $\xi_i$. Then $\fn$ is is annihilated by \f{\Op = p(x) \id} where \f{p(x) = \prod_{i=1}^\np (x-\xi_i)^\nders}.
\begin{proof}
For every test function $f(x)$ we have \[\langle f, \Op \fn \rangle = \langle p(x) f(x), \fn \rangle = \sum_{i=1}^\np \sum_{j=0}^{\nders-1} \langle  c_{ij}(x) p(x)f(x), \der{\delta}{j}(x-\xi_i) \rangle\] The function \f{r(x)=c_{ij}(x)p(x)f(x)} has a zero of  order $\nders$ at $x=\xi_i$. Therefore, all its derivatives up to \f{j} vanish at $\xi_i$ and so $\langle f, \Op \fn \rangle = 0$.
\end{proof}
\end{lem}

\begin{thm}\label{thm:piecewise-annihilating-operator}
Assume $\fn \in \rpaclass$ as in \eqref{eq:piecewise-full-def}, with $\Op$ of degree $\theorder$ annihilating every piece of $\fn$. Then the entire $\fn$ is annihilated as a distribution by the operator \f{\newop = \prod_{i=1}^{\np} (x-\xi_i)^{\theorder} \Op}.
\begin{proof}
By \eqref{eq:piecewise-full-def} and the fact that \f{\Op \widetilde{\fn_n}=0} we obtain
\begin{align*}
\Op \fn &= \Op \n{0}+\sum_{j=0}^\theorder \coeff_j(x) \partial^j \bigg\{ \sum_{i=1}^\np \n{i}(x) \step(x-\xi_i) \bigg\}\\
&= 0+ \sum_{i=1}^\np \sum_{j=0}^\theorder \coeff_j(x) \sum_{k=0}^j {j \choose k} \partial^{j-k} \big\{ \n{i}(x)\big\} \partial^k \big\{ \step(x-\xi_i) \big\}\\
&= \sum_{i=1}^\np \sum_{j=0}^\theorder \sum_{k=1}^j \coeff_j(x) {j \choose k} \partial^{j-k} \big\{\n{i}(x) \big\} \partial^k \big\{\step(x-\xi_i)\big\}+\sum_{i=1}^\np \step(x-\xi_i) \sum_{j=0}^\theorder \coeff_j(x) \partial^j\n{i}(x)\\
&= \sum_{i=1}^\np \sum_{k=0}^{\theorder-1} h_{ik}(x) \der{\delta}{k}(x-\xi_i)+\sum_{i=1}^\np \step(x-\xi_i) \Op \n{i}\\
&= \sum_{i=1}^\np \sum_{k=0}^{\theorder-1} h_{ik}(x) \der{\delta}{k}(x-\xi_i)
\end{align*}
where \f{h_{ik}(x)} is $k$-times differentiable at \f{\xi_i}. Then apply Lemma \ref{lem:deltader-annihilation}.
\end{proof}
\end{thm}

Theorem \ref{thm:piecewise-annihilating-operator} will later serve as the basis for recovering the locations of the discontinuities of \f{\fn \in \rpaclass}. The factor \f{\prod (x-\xi_i)^{\theorder}} effectively ``encodes'' the positions of the jumps into the operator itself. The ``decoding'' will then simply be to find these ``extra'' roots, once the ``enlarged'' operator is reconstructed.

The second form of the recurrence relation now immediately follows from Theorem \ref{thm:piecewise-annihilating-operator}.

\begin{thm}[Method II]\label{thm:piecewise-recurrence}
 The sequence of the moments of \f{\fn \in \rpaclass} satisfies the recurrence relation
\begin{align}\label{eq:piecewise-moments-recurrence-2}
 \bigl( \gdiffop(\shift) \cdot \dmop_{\newop}(k,\shift) \bigr) m_k = 0
\end{align}
where \f{\newop = \prod_{n=1}^\np (x-\xi_n)^\theorder \Op}.
\begin{proof}
 Theorem \ref{thm:main-recurrence-theorem} combined with Theorem \ref{thm:piecewise-annihilating-operator}.
\end{proof}
\end{thm}

\newcommand{\verynewop}{\ensuremath{\Op^{\bigstar}}}
\begin{lem}[Equivalence of Methods I and II]\label{lem:operator-multiplication-equivalence}
 Let \f{q(x)} be an arbitrary polynomial. Then
 \[\dmop_{q(x) \Op}(k,\shift)=q(\shift) \dmop_{\Op}(k,\shift)\]
\end{lem}

Equivalence of \eqref{eq:recurrence-entire-piecewise} and \eqref{eq:piecewise-moments-recurrence-2} then follows from Lemma \ref{lem:operator-multiplication-equivalence} by putting \f{q(x)=\prod_{i=1}^\np (x-\xi_i)^{\theorder}}.

To prove Lemma \ref{lem:operator-multiplication-equivalence} we need the following result.

\newcommand{\fj}{\ensuremath{\widehat{j}}} 
\begin{prop}\label{prop:mult-poly-rec}
 Let \f{p(x)=\sum_{i=0}^\alpha p_i x^i}, \f{q(x)=\sum_{i=0}^\beta q_i x^i} and \f{r(x)=p(x)q(x)=\sum_{i=0}^{\alpha+\beta} r_i x^i}. Let $\shift$ be the shift operator in $k$ and let $\fj$ be fixed. Then
\[
 \sum_{i=0}^{\alpha+\beta} r_i \ff{i+k}{\fj}\shift^{i-\fj} = p(\shift) \sum_{i=0}^\beta q_i \ff{i+k}{\fj} \shift^{i-\fj}
\]
\begin{proof}
 We extend the sequences of the coefficients \f{\{p_i\}} and \f{\{q_i\}} by zeros as necessary. By the rule of polynomial multiplication, \f{r_i=\sum_{j=0}^{\alpha} p_j q_{i-j}}. Now
\begin{align*}
 \sum_{i=0}^{\alpha+\beta} r_i \ff{i+k}{\fj}\shift^{i-\fj} &= \sum_{i=0}^{\alpha+\beta} \sum_{j=0}^{\alpha} p_j q_{i-j} \ff{i-j+k+j}{\fj}\shift^{i-j+\fj+j}
= \sum_{j=0}^{\alpha} p_j \shift^j \sum_{i=0}^{\alpha+\beta} q_{i-j} \ff{(i-j)+k}{\fj} \shift^{(i-j)+\fj}\\
\expln{i-j \rightarrow i} &= \sum_{j=0}^{\alpha} p_j \shift^j \sum_{i=-j}^{\alpha+\beta-j} q_i \ff{i+k}{\fj} \shift^{i-\fj} = p(\shift) \sum_{i=0}^\beta q_i \ff{i+k}{\fj} \shift^{i-\fj}
\qedhere
\end{align*}
\end{proof}
\end{prop}

\begin{proof}[Proof of Lemma \ref{lem:operator-multiplication-equivalence}]
\newcommand{\newpolycoeff}{\widehat{\polycoeff}}
Recall that \f{\dmop_{\Op}=\sum_{i,j} \polycoeff_{i,j} \gprc{i}{j}(k,\shift)} where \f{\gprc{i}{j}(k,\shift) = (-1)^j \ff{i+k}{j} \shift^{i-j}}, \f{\Op = \coeff_j(x) \partial^j} and \f{\coeff_j(x)=\sum_{i=0}^{\polymaxorder_j} \polycoeff_{i,j} x^i}. Now let \f{\newcoeff_j(x)=q(x)\coeff_j(x) = \sum \newpolycoeff_{i,j}x^i}, then
\begin{align*}
 \dmop_{q(x)\Op}(k,\shift) &= \sum_{j=0}^\theorder (-1)^j \sum_{i=0}^{\polymaxorder_j+\deg q} \newpolycoeff_{i,j} \ff{i+k}{j} \shift^{i-j}\\
\explntext{Proposition \ref{prop:mult-poly-rec}} &= \sum_{j=0}^\theorder (-1)^j q(\shift) \sum_{i=0}^{\polymaxorder_j} \polycoeff_{i,j} \ff{i+k}{j} \shift^{i-j}\\
&= q(\shift) \dmop_{\Op}(k,\shift) \qedhere
\end{align*}
\end{proof}

\subsubsection{Different operators on each interval}
Recall that we want to find a recurrence relation for the sequence \f{m_k=\sum m_{k,n}} where each subsequence is annihilated by the difference operator
\[
 \annop_n = \ggdiffop{\theorder_n}{\xi_n}{\xi_{n+1}} \cdot \dmop_{\Op_n}(k,\shift)
\]
There exist at least two approaches, both of which involve techniques from the theory of non-commutative polynomials -- the so called Ore polynomial rings (see \cite{ore1933tnc}). Both differential and difference operators with polynomial coefficients are members of the appropriate Ore algebra. The \emph{least common left multiple} (LCLM) of two polynomials \f{p,q} is the unique polynomial $r$ of minimal degree such that both $p$ and $q$ are right-hand factors of $r$, i.e. \f{r=p'p=q'q} for some polynomials \f{p',q'}. The LCLM may be explicitly found by the non-commutative version of the polynomial division algorithm. The complete theory may be found in \cite{ore1933tnc}.
\begin{enumerate}[\bfseries {Approach} I]
 \item Given the operators \f{\annop_n} such that \f{\annop_n m_{k,n} \equiv 0}, the operator \f{\annop^{\flat}} which annihilates the sum \f{\sum_n m_{k,n}} is given by the least common left multiple of \f{\{\annop_n\}}.
\item Given the operators \f{\Op_n} which annihilate the pieces \f{\fn_n} separately, the operator \f{\Op^{\dagger}} which annihilates every piece simultaneously is the least common left multiple of \f{\{\Op_n\}}. Then the annihilating operator for \f{\{m_k\}} is \f{\annop^{\dagger} = \ggdiffop{\theorder^{\dagger}}{a}{b} \dmop_{\Op^{\dagger}}(k,\shift)}.
\end{enumerate}

We are not aware of any general procedure by which the coefficients of \f{\annop^{\dagger}} or \f{\annop^{\flat}} may be related to the coefficients of each \f{\Op_n} in some tractable manner, unless the operators \f{\Op_n} commute.

\begin{example}[Piecewise sinusoids]\label{ex:piecewise-sin-annihilating-op}
 Let $\fn$ consist of two sinusoid pieces: \f{\fn_1(x)=c_1 \sin(\omega_1 x+\phi_1)} and \f{\fn_2(x)=c_2 \sin(\omega_2 x+\phi_2)} with break point $\xi$. The annihilating operators for $\fn_1$ and $\fn_2$ are, respectively, \f{\Op_1=\partial^2 + \omega_1^2 \id} and \f{\Op_2=\partial^2 + \omega_2^2 \id}. The operator \f{\Op^{\dagger}=(\partial^2+\omega_1^2 \id) \cdot (\partial^2+\omega_2^2 \id)} annihilates both pieces simultaneously, therefore \f{(x-\xi)^4 \Op^{\dagger}} annihilates the entire $\fn$.
\qed
\end{example}

\section{Moment inversion}\label{sec:inversion}
We now present our method for moment inversion for piecewise $D$-finite functions. Recall that our purpose is to reconstruct the parameters
\begin{equation}\label{eq:params-def}
\pars \isdef \bigl\{\{\polycoeff_{i,j}\},\{\xi_i\},\{\alpha_{i,n}\}\bigr\}
\end{equation}
from the input
\begin{equation}\label{eq:input-def}
\inp \isdef \bigl\{\{m_k\},\theorder, \{\polymaxorder_j\}, \np, a,b \bigr\}
\end{equation}
First we establish explicit connections between $\pars$ and $\inp$ -- the ``forward mapping'' \f{\fwm:\pars \to \inp}. Then we derive the inverse mapping \f{\ivm=\fwm^{-1}} and provide simple conditions for the solvability of the resulting inverse systems.

\subsection{Forward equations}\label{sec:forward-equations}
\newcommand{\polyder}[2][\fn]{\ensuremath{\mathfrak{g}^{#1}_{#2}}}

Recall the polynomials \f{\prc(k,\shift)} which were defined in \eqref{eq:mmm-def} during the derivation of the recurrence relation \eqref{eq:main-recurrence-explicit}. Given a multiindex \f{(i,j)} we define the ``shifted'' moment sequence
\begin{align}\label{eq:mmm-seq-def}
 \mmm &\isdef \left(\gdiffop(\shift) \cdot \prc(k,\shift) \right) m_k 
\end{align}

For each \f{j=0,\dotsc,\theorder} let \f{h_j(z)} be the formal power series
\begin{align*}
 h_j(z) &\isdef \sum_{k=0}^\infty \mm{0}{j}{k} z^k
\end{align*}

Finally, for any \f{\fn(x)} let
\begin{align*}
\polyder{j}(x) &\isdef \gdiffop(x) \frac{d^j}{dx^j}\fn(x)\\
I_{\fn}(z) &\isdef \sum_{k=0}^\infty m_k(\fn) z^k
\end{align*}
\f{I_{\fn}(z)} is called the \emph{moment-generating function} of $\fn$.

\begin{thm}\label{thm:operator-coeffs-main-thm}
Let \f{\fn \in \rpaclass} be annihilated by $\Op$ (either in the usual sense if \f{\np=0} or as a distribution if \f{\np >0}, in which case \f{\Op = \prod_{i=1}^{\np} (x-\xi_i)^{\theorder} \Op^{\dagger}} with $\Op^{\dagger}$ annihilating every piece) , where \f{\Op = \sum_{j=0}^\theorder \coeff_j(x) \partial^j} and \f{\coeff_j(x)=\sum_{i=0}^{\polymaxorder_j} \polycoeff_{i,j} x^i}. Then
\begin{enumerate}[\bfseries (A)]
\item\label{part:systemH} The vector \f{\vec{\polycoeff}=(\polycoeff_{i,j})} satisfies a linear homogeneous system
\begin{align}\label{eq:systemH}
H \vec{\polycoeff} = \begin{pmatrix}
\mm{0}{0}{0} & \mm{1}{0}{0} & \dots & \mm{\polymaxorder_{\theorder}}{\theorder}{0} \\
\mm{0}{0}{1} & \mm{1}{0}{1} & \dots & \mm{\polymaxorder_{\theorder}}{\theorder}{1} \\
\vdots & \vdots & \vdots & \vdots \\
\mm{0}{0}{\nmH} & \mm{1}{0}{\nmH} & \dots & \mm{\polymaxorder_{\theorder}}{\theorder}{\nmH} \\
\end{pmatrix} \begin{pmatrix}
\polycoeff_{0,0} \\
\polycoeff_{1,0} \\
\vdots \\
\polycoeff_{\polymaxorder_{\theorder},\theorder} \\
\end{pmatrix}= 0
\end{align}
for all \f{\nmH \in \naturals}.
\item\label{part:mgf} \f{\mmm=m_{i+k}\left( \polyder{j}(x) \right)} for all \f{0 \leq j \leq \theorder}, \f{0 \leq i \leq \polymaxorder_j} and \f{k \in \naturals} (the moments are taken in \f{[a,b]}). Consequently, \f{h_j(z)} is the moment generating function of \f{\polyder{j}(x)}.
\item\label{part:hermite-pade} The functions \f{\{1,h_0(z), \dotsc h_{\theorder}(z)\}} are polynomially dependent:
\begin{align*}
 q(z) + \sum_{j=0}^{\theorder} h_j(z) p^*_j(z) &= 0
\end{align*}
where \f{p^*_j(z)=z^{\max \polymaxorder_j} \coeff_j(z^{-1})} and \f{q(z)} is a polynomial with \f{\deg q < \max \polymaxorder_j}.
\end{enumerate}
\end{thm}

\begin{proof}[Proof of \ref{part:systemH}]
 By Theorem \ref{thm:piecewise-recurrence} and \eqref{eq:mmm-seq-def}we have
\[
 \sum_{j=0}^\theorder \sum_{i=0}^{\polymaxorder_j} \polycoeff_{i,j} \mmm = 0, \quad k=0,1,\dotsc
\]
This is exactly \eqref{eq:systemH}.
\end{proof}

\begin{prop}\label{prop:application-of-polynomial-to-moments}
 Let $p(x)$ be a polynomial in $x$. Then for every $f(x)$
\[
 m_k\bigl(p(x) f(x) \bigr) = p(\shift) m_k(f(x))
\]
\begin{proof}
 Let \f{p(x) = x^r}. Then
\[
 m_k\bigl(x^r f(x)\bigr) = \int_a^b x^k x^r f(x) dx = \int_a^b x^{k+r} f(x)dx = m_{k+r} (f(x)) = \shift^r m_k(f(x))
\]
The proof for an arbitrary polynomial \f{p(x)} follows by linearity.
\end{proof}
\end{prop}

\begin{proof}[Proof of \ref{part:mgf}]
 Fix \f{j \leq \theorder}, \f{i \leq \max \polymaxorder_j} and define \f{\Op_{ij} = x^i \partial^j}. By \eqref{eq:diffmoments-defs} and \eqref{eq:diffmoments-linear-term} we have
\begin{align}\label{eq:partBproof-decomposition}
 m_k(\Op_{ij} \fn) = \dmop_{\Op_{ij}}(k,\shift)m_k(\fn) + \badpart_k(\Op_{ij},\fn)
\end{align}
By \eqref{eq:mmm-def} we have: \f{\dmop_{\Op_{ij}}(k,\shift)=\gprc{i}{j}(k,\shift)}. By Theorem \ref{thm:main-differences-generalized} we have \f{\gdiffop(\shift) \badpart_k(\Op_{ij},\fn)=0}. Finally,
\begin{align*}
 \mmm &= \bigl(\gdiffop(\shift) \cdot \prc(k,\shift)\bigr)m_k(\fn)
= \gdiffop(\shift) \dmop_{\Op_{ij}}(k,\shift)m_k\\
\expln{\ref{eq:partBproof-decomposition}} &= \gdiffop(\shift) m_k(\Op_{ij}\fn) - \gdiffop(\shift) \badpart_k(\Op_{ij},\fn)
= \gdiffop(\shift) m_k(\Op_{ij}\fn)\\
\explntext{Proposition \ref{prop:application-of-polynomial-to-moments}} &= m_k(\diffop(x) x^i \partial^j \fn) = m_{i+k}(\polyder{j}(x))
\qedhere 
\end{align*}
\end{proof}

\begin{proof}[Proof of \ref{part:hermite-pade}]
\newcommand{\themax}{\ensuremath{\polymaxorder^*}}
 Let \f{\themax = \max \polymaxorder_j}. We have \f{p^*_j(z)=\sum_{i=0}^{\polymaxorder_j} a_{i,j}z^{\themax-i}}. Denote the power series
\[
        q(z)=-\sum_{j=0}^\theorder h_j(z) p^*_j(z)=\sum_{k=0}^\infty q_k z^k
\]
An immediate consequence of \ref{part:mgf} is that \f{\mm{i'}{j}{k'}=\mm{i}{j}{k}} for \f{i'+k'=i+k} and all $j$. Then for all \f{m \geq \themax} 
\[
 -q_{m} = \sum_{j=0}^\theorder \sum_{i=0}^{\polymaxorder_j} \polycoeff_{i,j} \mm{0}{j}{m-\themax+i}= \sum_{j=0}^\theorder \sum_{i=0}^{\polymaxorder_j} \polycoeff_{i,j} \mm{i}{j}{m-\themax} = 0
\]
So \f{q(z)} is a polynomial of degree at most \f{\themax-1}. This completes the proof of Theorem \ref{thm:operator-coeffs-main-thm}.
\end{proof}

\begin{rem}\label{rem:hankel-striped}
$H$ has the structure of \emph{Hankel-striped} matrix \f{H = [V_0 \dotsc V_{\theorder}]} where each ``stripe'' is a Hankel matrix
\begin{equation}\label{eq:stripe-def}
 V_j = \begin{bmatrix}
\mm{0}{j}{0} & \mm{1}{j}{0} & \dots & \mm{\polymaxorder_{j}}{j}{0} \\
\mm{0}{j}{1} & \mm{1}{j}{1} & \dots & \mm{\polymaxorder_{j}}{j}{1} \\
\vdots & \vdots & \vdots & \vdots \\
\mm{0}{j}{\nmH} & \mm{1}{j}{\nmH} & \dots & \mm{\polymaxorder_{j}}{j}{\nmH} \\
\end{bmatrix}
\end{equation}

Hankel-striped matrices appear as central objects in contexts such as \emph{Hermite-Pad\'{e}} approximation (the standard Pad\'{e} approximation being its special case), minimal realization problem in control theory and Reed-Solomon codes (\cite{labahn1992icb,kalman:mpr,Byrnes83onthe}). In fact, the system of polynomials \f{\{p^*_j(z)\}} is called the Pad\'{e}-Hermite form for \f{\Phi=\{1,h_0(z),\dotsc,h_{\theorder}(z)\}}.
\end{rem}

\begin{rem}
Moment-generating functions are a powerful tool for the investigation of the properties of the sequence $m_k$. For instance, the asymptotic behavior of the general term may be derived from the analytic properties of $I_{\fn}(z)$ (\cite{flajolet2005ac}).
\end{rem}

Now suppose the operator \f{\Op} annihilating every piece $\fn_n$ of \f{\fn \in \rpaclass}, as well as the jump points \f{\{\xi_n\}}, are known. Let \f{\{\bas_i\}_{i=1}^\theorder} be a basis for the space \f{\nullsp{\Op}}. Then \f{\fn_n(x)=\sum_{i=1}^\theorder \alpha_{in} \bas_i(x)}. Applying the moment transform to both sides of the last equation and summing over \f{n=0,\dotsc,\np} gives

\begin{prop}\label{prop:particular-solution}
Denote \f{c_{i,k}^n=\int_{\xi_n}^{\xi_{n+1}} x^k u_i(x)} for \f{n=0,\dotsc,\np}. Then \f{\forall \nmC \in \naturals}:
\begin{align}\label{eq:particular-solution}
\begin{pmatrix}
c_{1,0}^0 & \dotsc & c_{\theorder,0}^0 & \dotsc & c_{\theorder,0}^\np\\
\vdots & \vdots & \vdots & \vdots  & \vdots\\
c_{1,\nmC}^0 & \dotsc & c_{\theorder,\nmC}^0 & \dotsc & c_{\theorder,\nmC}^\np
\end{pmatrix}
\begin{pmatrix}
 \alpha_{1,0}\\ \vdots \\ \alpha_{\theorder,0}\\
\vdots \\ \alpha_{\theorder,\np}
\end{pmatrix}
=
\begin{pmatrix}
 m_0\\ m_1\\ \vdots\\ m_{\nmC}
\end{pmatrix}
\qed
\end{align}
\end{prop}

\newcommand{\recOp}{\ensuremath{\widetilde{\Op}}}
\newcommand{\qjumps}{\f{\np>0}?}
\newcommand{\theinput}{Input: \f{\{m_k\}_{k=0}^\nm}, \f{\theorder, \{\polymaxorder_j\}}, \np}
\newcommand{\substituteone}{\f{\polymaxorder_j \leftarrow \polymaxorder_j+\theorder\np}}
\newcommand{\substitutetwo}{\f{\bigl(\Op \leftarrow \prod_{i=1}^\np (x-\xi_i)^\theorder \Op\bigr)}}
\newcommand{\oprecone}{\underline{Reconstruct \f{\Op}}}
\newcommand{\oprectwo}{Solve \f{H\vec{a}=0 \quad\eqref{eq:systemH}}}
\newcommand{\oprecthree}{\f{\recOp = \Op_{\vec{a}}}}
\newcommand{\jumprecone}{\underline{Recover jump points}}
\newcommand{\jumprectwo}{\f{\recOp \rightarrow \f{\{\xi_n\}, \Op^{\dagger}}}}
\newcommand{\solrecone}{\underline{Recover particular solution(s)}}
\newcommand{\solrectwo}{Find basis for \f{\nullsp{\Op^{\dagger}}}}
\newcommand{\solrecthree}{Solve \f{C \vec{\alpha} = \vec{m}} \eqref{eq:particular-solution}}
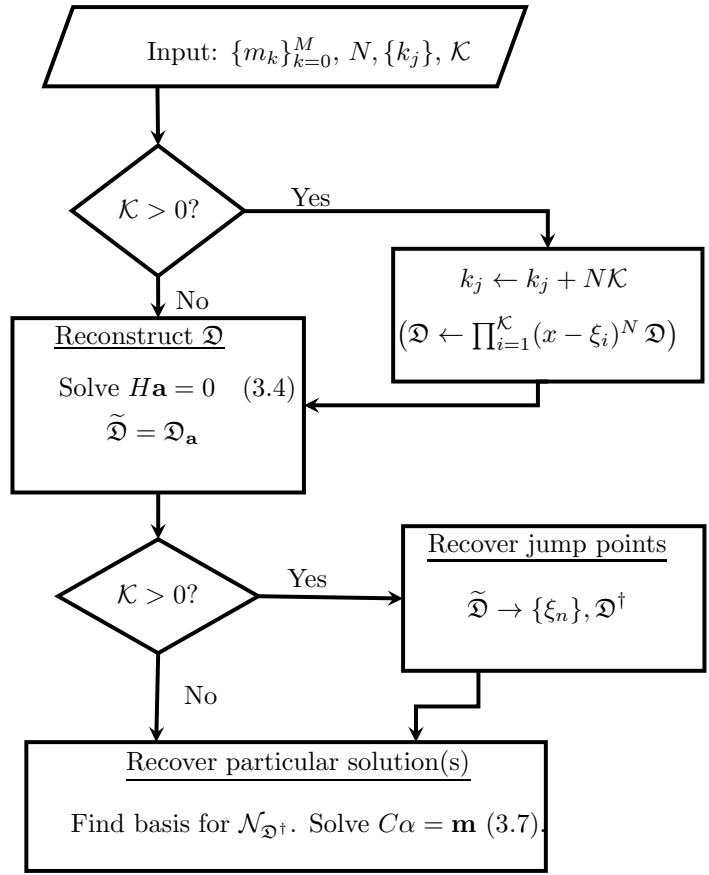
\begin{figure}[h]
\centering
\ifx\du\undefined
  \newlength{\du}
\fi
\setlength{\du}{15\unitlength}
\begin{tikzpicture}
\pgftransformxscale{1.000000}
\pgftransformyscale{-1.000000}
\definecolor{dialinecolor}{rgb}{0.000000, 0.000000, 0.000000}
\pgfsetstrokecolor{dialinecolor}
\definecolor{dialinecolor}{rgb}{1.000000, 1.000000, 1.000000}
\pgfsetfillcolor{dialinecolor}
\definecolor{dialinecolor}{rgb}{1.000000, 1.000000, 1.000000}
\pgfsetfillcolor{dialinecolor}
\fill (5.327940\du,0.350000\du)--(16.750000\du,0.350000\du)--(16.022060\du,2.350000\du)--(4.600000\du,2.350000\du)--cycle;
\pgfsetlinewidth{0.100000\du}
\pgfsetdash{}{0pt}
\pgfsetdash{}{0pt}
\pgfsetmiterjoin
\definecolor{dialinecolor}{rgb}{0.000000, 0.000000, 0.000000}
\pgfsetstrokecolor{dialinecolor}
\draw (5.327940\du,0.350000\du)--(16.750000\du,0.350000\du)--(16.022060\du,2.350000\du)--(4.600000\du,2.350000\du)--cycle;
\definecolor{dialinecolor}{rgb}{0.000000, 0.000000, 0.000000}
\pgfsetstrokecolor{dialinecolor}
\node[anchor=west] at (6.920000\du,1.492500\du){\theinput};
\definecolor{dialinecolor}{rgb}{1.000000, 1.000000, 1.000000}
\pgfsetfillcolor{dialinecolor}
\fill (7.471815\du,3.842520\du)--(9.645520\du,5.490103\du)--(7.471815\du,7.137687\du)--(5.298110\du,5.490103\du)--cycle;
\pgfsetlinewidth{0.100000\du}
\pgfsetdash{}{0pt}
\pgfsetdash{}{0pt}
\pgfsetmiterjoin
\definecolor{dialinecolor}{rgb}{0.000000, 0.000000, 0.000000}
\pgfsetstrokecolor{dialinecolor}
\draw (7.471815\du,3.842520\du)--(9.645520\du,5.490103\du)--(7.471815\du,7.137687\du)--(5.298110\du,5.490103\du)--cycle;
\definecolor{dialinecolor}{rgb}{0.000000, 0.000000, 0.000000}
\pgfsetstrokecolor{dialinecolor}
\node[anchor=west] at (6.2\du,5.45\du){\qjumps};
\pgfsetlinewidth{0.100000\du}
\pgfsetdash{}{0pt}
\pgfsetdash{}{0pt}
\pgfsetbuttcap
{
\definecolor{dialinecolor}{rgb}{0.000000, 0.000000, 0.000000}
\pgfsetfillcolor{dialinecolor}
\pgfsetarrowsend{stealth}
\definecolor{dialinecolor}{rgb}{0.000000, 0.000000, 0.000000}
\pgfsetstrokecolor{dialinecolor}
\draw (7.455510\du,2.350000\du)--(7.471815\du,3.842520\du);
}
\definecolor{dialinecolor}{rgb}{1.000000, 1.000000, 1.000000}
\pgfsetfillcolor{dialinecolor}
\fill (13.400000\du,6.450000\du)--(13.400000\du,9.800000\du)--(21.200000\du,9.800000\du)--(21.200000\du,6.450000\du)--cycle;
\pgfsetlinewidth{0.100000\du}
\pgfsetdash{}{0pt}
\pgfsetdash{}{0pt}
\pgfsetmiterjoin
\definecolor{dialinecolor}{rgb}{0.000000, 0.000000, 0.000000}
\pgfsetstrokecolor{dialinecolor}
\draw (13.400000\du,6.450000\du)--(13.400000\du,9.800000\du)--(21.200000\du,9.800000\du)--(21.200000\du,6.450000\du)--cycle;
\definecolor{dialinecolor}{rgb}{0.000000, 0.000000, 0.000000}
\pgfsetstrokecolor{dialinecolor}
\node[anchor=west] at (14.805000\du,7.267500\du){\substituteone};
\node[anchor=west] at (13.2\du,8.6\du){\substitutetwo};

\definecolor{dialinecolor}{rgb}{0.000000, 0.000000, 0.000000}
\pgfsetstrokecolor{dialinecolor}
\node[anchor=west] at (10.550000\du,5.150000\du){Yes};
\definecolor{dialinecolor}{rgb}{1.000000, 1.000000, 1.000000}
\pgfsetfillcolor{dialinecolor}
\fill (3.800000\du,8.200000\du)--(3.800000\du,12.600000\du)--(11.150000\du,12.600000\du)--(11.150000\du,8.200000\du)--cycle;
\pgfsetlinewidth{0.100000\du}
\pgfsetdash{}{0pt}
\pgfsetdash{}{0pt}
\pgfsetmiterjoin
\definecolor{dialinecolor}{rgb}{0.000000, 0.000000, 0.000000}
\pgfsetstrokecolor{dialinecolor}
\draw (3.800000\du,8.200000\du)--(3.800000\du,12.600000\du)--(11.150000\du,12.600000\du)--(11.150000\du,8.200000\du)--cycle;
\definecolor{dialinecolor}{rgb}{0.000000, 0.000000, 0.000000}
\pgfsetstrokecolor{dialinecolor}
\node[anchor=west] at (4.6\du,8.7\du){\oprecone};
\node[anchor=west] at (4.7\du,10.0\du){\oprectwo};
\node[anchor=west] at (5.9\du,11.0\du){\oprecthree};
\pgfsetlinewidth{0.100000\du}
\pgfsetdash{}{0pt}
\pgfsetdash{}{0pt}
\pgfsetbuttcap
{
\definecolor{dialinecolor}{rgb}{0.000000, 0.000000, 0.000000}
\pgfsetfillcolor{dialinecolor}
\pgfsetarrowsend{stealth}
\definecolor{dialinecolor}{rgb}{0.000000, 0.000000, 0.000000}
\pgfsetstrokecolor{dialinecolor}
\draw (7.471815\du,7.137687\du)--(7.475000\du,8.200000\du);
}
\pgfsetlinewidth{0.100000\du}
\pgfsetdash{}{0pt}
\pgfsetdash{}{0pt}
\pgfsetmiterjoin
\pgfsetbuttcap
{
\definecolor{dialinecolor}{rgb}{0.000000, 0.000000, 0.000000}
\pgfsetfillcolor{dialinecolor}
\pgfsetarrowsend{stealth}
{\pgfsetcornersarced{\pgfpoint{0.000000\du}{0.000000\du}}\definecolor{dialinecolor}{rgb}{0.000000, 0.000000, 0.000000}
\pgfsetstrokecolor{dialinecolor}
\draw (9.645520\du,5.490103\du)--(9.645520\du,5.500000\du)--(17.300000\du,5.500000\du)--(17.300000\du,6.450000\du);
}}
\pgfsetlinewidth{0.100000\du}
\pgfsetdash{}{0pt}
\pgfsetdash{}{0pt}
\pgfsetmiterjoin
\pgfsetbuttcap
{
\definecolor{dialinecolor}{rgb}{0.000000, 0.000000, 0.000000}
\pgfsetfillcolor{dialinecolor}
\pgfsetarrowsend{stealth}
{\pgfsetcornersarced{\pgfpoint{0.000000\du}{0.000000\du}}\definecolor{dialinecolor}{rgb}{0.000000, 0.000000, 0.000000}
\pgfsetstrokecolor{dialinecolor}
\draw (17.300000\du,9.800000\du)--(17.050000\du,9.800000\du)--(17.050000\du,10.400000\du)--(11.150000\du,10.400000\du);
}}
\definecolor{dialinecolor}{rgb}{0.000000, 0.000000, 0.000000}
\pgfsetstrokecolor{dialinecolor}
\node[anchor=west] at (7.650000\du,7.750000\du){No};
\definecolor{dialinecolor}{rgb}{1.000000, 1.000000, 1.000000}
\pgfsetfillcolor{dialinecolor}
\fill (7.474372\du,13.780000\du)--(9.998745\du,15.217795\du)--(7.474372\du,16.655589\du)--(4.950000\du,15.217795\du)--cycle;
\pgfsetlinewidth{0.100000\du}
\pgfsetdash{}{0pt}
\pgfsetdash{}{0pt}
\pgfsetmiterjoin
\definecolor{dialinecolor}{rgb}{0.000000, 0.000000, 0.000000}
\pgfsetstrokecolor{dialinecolor}
\draw (7.474372\du,13.780000\du)--(9.998745\du,15.217795\du)--(7.474372\du,16.655589\du)--(4.950000\du,15.217795\du)--cycle;
\definecolor{dialinecolor}{rgb}{0.000000, 0.000000, 0.000000}
\pgfsetstrokecolor{dialinecolor}
\node[anchor=west] at (6.164372\du,15.160295\du){\qjumps};
\pgfsetlinewidth{0.100000\du}
\pgfsetdash{}{0pt}
\pgfsetdash{}{0pt}
\pgfsetbuttcap
{
\definecolor{dialinecolor}{rgb}{0.000000, 0.000000, 0.000000}
\pgfsetfillcolor{dialinecolor}
\pgfsetarrowsend{stealth}
\definecolor{dialinecolor}{rgb}{0.000000, 0.000000, 0.000000}
\pgfsetstrokecolor{dialinecolor}
\draw (7.475000\du,12.600000\du)--(7.474372\du,13.780000\du);
}
\definecolor{dialinecolor}{rgb}{1.000000, 1.000000, 1.000000}
\pgfsetfillcolor{dialinecolor}
\fill (13.650000\du,13.450000\du)--(13.650000\du,17.100000\du)--(21.250000\du,17.100000\du)--(21.250000\du,13.450000\du)--cycle;
\pgfsetlinewidth{0.100000\du}
\pgfsetdash{}{0pt}
\pgfsetdash{}{0pt}
\pgfsetmiterjoin
\definecolor{dialinecolor}{rgb}{0.000000, 0.000000, 0.000000}
\pgfsetstrokecolor{dialinecolor}
\draw (13.650000\du,13.450000\du)--(13.650000\du,17.100000\du)--(21.250000\du,17.100000\du)--(21.250000\du,13.450000\du)--cycle;
\definecolor{dialinecolor}{rgb}{0.000000, 0.000000, 0.000000}
\pgfsetstrokecolor{dialinecolor}
\node[anchor=west] at (14.0\du,14.0\du){\jumprecone};
\node[anchor=west] at (15.0\du,15.5\du){\jumprectwo};
\definecolor{dialinecolor}{rgb}{1.000000, 1.000000, 1.000000}
\pgfsetfillcolor{dialinecolor}
\fill (4.150000\du,18.900000\du)--(4.150000\du,22.150000\du)--(17.250000\du,22.150000\du)--(17.250000\du,18.900000\du)--cycle;
\pgfsetlinewidth{0.100000\du}
\pgfsetdash{}{0pt}
\pgfsetdash{}{0pt}
\pgfsetmiterjoin
\definecolor{dialinecolor}{rgb}{0.000000, 0.000000, 0.000000}
\pgfsetstrokecolor{dialinecolor}
\draw (4.150000\du,18.900000\du)--(4.150000\du,22.150000\du)--(17.250000\du,22.150000\du)--(17.250000\du,18.900000\du)--cycle;
\definecolor{dialinecolor}{rgb}{0.000000, 0.000000, 0.000000}
\pgfsetstrokecolor{dialinecolor}
\node[anchor=west] at (6.393750\du,19.5\du){\solrecone};
\node[anchor=west] at (5.0\du,21.0\du){\solrectwo. \solrecthree.};
\pgfsetlinewidth{0.100000\du}
\pgfsetdash{}{0pt}
\pgfsetdash{}{0pt}
\pgfsetbuttcap
{
\definecolor{dialinecolor}{rgb}{0.000000, 0.000000, 0.000000}
\pgfsetfillcolor{dialinecolor}
\pgfsetarrowsend{stealth}
\definecolor{dialinecolor}{rgb}{0.000000, 0.000000, 0.000000}
\pgfsetstrokecolor{dialinecolor}
\draw (7.474372\du,16.655589\du)--(7.425000\du,18.900000\du);
}
\definecolor{dialinecolor}{rgb}{0.000000, 0.000000, 0.000000}
\pgfsetstrokecolor{dialinecolor}
\node[anchor=west] at (10.460000\du,14.757500\du){Yes};
\definecolor{dialinecolor}{rgb}{0.000000, 0.000000, 0.000000}
\pgfsetstrokecolor{dialinecolor}
\node[anchor=west] at (7.860000\du,17.707500\du){No};
\pgfsetlinewidth{0.100000\du}
\pgfsetdash{}{0pt}
\pgfsetdash{}{0pt}
\pgfsetmiterjoin
\pgfsetbuttcap
{
\definecolor{dialinecolor}{rgb}{0.000000, 0.000000, 0.000000}
\pgfsetfillcolor{dialinecolor}
\pgfsetarrowsend{stealth}
{\pgfsetcornersarced{\pgfpoint{0.000000\du}{0.000000\du}}\definecolor{dialinecolor}{rgb}{0.000000, 0.000000, 0.000000}
\pgfsetstrokecolor{dialinecolor}
\draw (15.550000\du,17.100000\du)--(15.550000\du,18.000000\du)--(13.975000\du,18.000000\du)--(13.975000\du,18.900000\du);
}}
\pgfsetlinewidth{0.100000\du}
\pgfsetdash{}{0pt}
\pgfsetdash{}{0pt}
\pgfsetbuttcap
{
\definecolor{dialinecolor}{rgb}{0.000000, 0.000000, 0.000000}
\pgfsetfillcolor{dialinecolor}
\pgfsetarrowsend{stealth}
\definecolor{dialinecolor}{rgb}{0.000000, 0.000000, 0.000000}
\pgfsetstrokecolor{dialinecolor}
\draw (9.998745\du,15.217795\du)--(13.650000\du,15.275000\du);
}
\end{tikzpicture}
\caption{The reconstruction algorithm}
 \label{fig:reconstruction-algorithm}
\end{figure}

\subsection{The inversion algorithm}\label{sec:thealg}
The reconstruction algorithm which is based on the results of the previous section is depicted schematically in Figure \ref{fig:reconstruction-algorithm}. The solvability of the corresponding systems is discussed in the next section. Note the following:

\begin{enumerate}[(a)]
 \item At the initial stage,  the ``encoding'' of the (yet unknown) jump points takes place. In practice, this means the ``enlargement'' of $\Op$ to \f{\newop \isdef \prod_{i=1}^\np (x-\xi_i)^\theorder \Op}. The parameters of the problem therefore change as follows: $\theorder$ remains the same while \f{\polymaxorder_j \leftarrow \polymaxorder_j+\theorder\np} - see \eqref{eq:piecewise-moments-recurrence-2}.
\item  By Theorem \ref{thm:piecewise-annihilating-operator}, \f{\{\xi_n\}} are the $\np$ distinct common roots of the polynomials which are the coefficients of $\newop$ of multiplicity $\theorder$. The remaining part of the coefficients define the operator \f{\Op^{\dagger}} which annihilates every piece of $\fn$.
\end{enumerate}

\begin{example}[Examples \ref{ex:direct-exponential} and \ref{ex:exponential-recurrence} continued]\label{ex:firstorder}
\f{\fn(x) = \alpha e^{\beta x}} on \f{[0,1]} is annihilated by \f{\Op = \partial - \beta \id}.

Writing down \eqref{eq:systemH} with \f{\nmH=0} yields
\begin{align*}
 \begin{bmatrix}
  m_2-m_1 & -(2m_1-m_0)
 \end{bmatrix}
 \begin{bmatrix}
  -\beta \\1
 \end{bmatrix}
 & = 0
\end{align*}
which has the solution
\begin{align*}
\beta=\frac{2m_1-m_0}{m_1-m_2}
\end{align*}
The constant $\alpha$ is then recovered by
\[
\alpha = \frac{m_0(\fn)}{m_0(e^{\beta x})}=\frac{\beta m_0}{e^{\beta}-1}
\]
Note that this solution requires the first 3 moments instead of 2 as in \eqref{eq:constsystem1sol1}. \qed
\end{example}

Additional examples of complete inversion procedures are elaborated in Appendix \ref{apx:experiments}.

\subsection{Solvability of inverse equations}\label{sec:solvability}
The constants $\nmH$ and $\nmC$ determine the minimal size of the corresponding linear systems \eqref{eq:systemH} and \eqref{eq:particular-solution} in order for all the solutions of these systems to be also solutions of the original problem.

\begin{thm}\label{thm:sizeH}
If $\vec{b} \in \nullsp{H}$, then
\begin{align*}
 m_k\bigl(\diffop(x) \bigl(\Op_{\vec{b}} \fn \bigr)(x) \bigr) = 0 \qquad k=0,1,\dots,\nmH
\end{align*}
\begin{proof}

Denote \f{\vec{b}=(b_{ij})} and let \f{k \geq 0}. By Theorem \ref{thm:operator-coeffs-main-thm} Part \ref{part:mgf} and Proposition \ref{prop:application-of-polynomial-to-moments}, the product of the $k+1$-st row of $H$ with $\vec{b}$ is
\begin{align*}
0=\sum_{j=0}^{\theorder} \sum_{i=0}^{\polymaxorder_j} b_{ij} \mmm &= \sum_{j=0}^{\theorder} \sum_{i=0}^{\polymaxorder_j} b_{ij} m_{i+k}(\polyder{j}(x)) = \sum_{j=0}^{\theorder} \sum_{i=0}^{\polymaxorder_j} m_k(b_{ij} x^i \polyder{j}(x))
= m_k \Biggl(\sum_{j=0}^{\theorder} \biggl( \sum_{i=0}^{\polymaxorder_j} b_{ij} x^i \biggr) \polyder{j}(x) \Biggr)\\&=m_k \biggl(\diffop(x) \bigl(\Op_{\vec{b}} \fn\bigr)(x)\biggr)
\qedhere
\end{align*}
\end{proof}
\end{thm}

If it is possible to estimate how many moments of \f{F=\diffop(x) \bigl(\Op_{\vec{b}} \fn \bigr)(x)} should vanish in order to guarantee the identical vanishing of $F$ and therefore also of \f{\Op_{\vec{b}}\fn} for \emph{all possible differential operators} of the prescribed complexity (i.e. of given order and given degrees of its coefficients), then $\nmH$ may be taken to be this number. Then, every solution of \eqref{eq:systemH} will correspond to some annihilating operator. For any specific $\fn$ and $\Op_{\vec{b}}$ such a finite number exists, because any nonzero piecewise-continuous integrable  function has at least some nonzero moments. However, this number may be arbitrarily large as shown by the next example.

\begin{example}[Legendre orthogonal polynomials]\label{ex:orthopoly}
It is known that every family of orthogonal polynomials satisfies a differential equation of order \f{\theorder=2} of the following type: 
\begin{align}\label{eq:orthpoly-diffeq}
\Op = q(x)\partial^2 + p(x)\partial + \lambda_n \id
\end{align}
where \f{q,p} are fixed polynomials with \f{\deg q \leq 2}, \f{\deg p \leq 1} and $\lambda_n$ is a scalar which is different for each member of the family. 

Consider \f{\{L_n(x)\}} - the family of Legendre orthogonal polynomials. The interval of orthogonality is \f{[a,b]=[-1,1]} and \f{\diffop=(\shift^2-\id)^2}. $L_n$ is annihilated by \f{\Op_n=(1-x^2)\partial^2-2x\partial+\lambda_n\id}. Furthermore, \f{\langle L_n(x), x^k \rangle = 0} for \f{k \leq n-1}. The ``reconstruction problem'' for $L_n(x)$ (assume it is normalized)  is to find the constant $\lambda_n$ from the moments (it is well-known that \f{\lambda_n=n(n+1)}).

Take an arbitrary vector \f{\vec{b}=\begin{bmatrix} b_{00} & b_{11} & b_{02} & b_{22}\end{bmatrix}^T} so \f{\Op_{\vec{b}}=b_{00}\id+b_{11}x\partial+(b_{02}+b_{22}x^2)\partial^2}. The function \f{\diffop(x) \Op_{\vec{b}} L_n} is a polynomial of degree \f{n+4}, so it is uniquely determined by its \f{n+5} moments (\cite{sig_ack}). By Theorem \ref{thm:sizeH}, this would be the minimal size of the system \eqref{eq:systemH}. The entries of the $k$-th row of $H$ are
\begin{align*}
 \mm{0}{0}{k} &= m_{k+4}-2m_{k+2}+m_k & i=0,j=0\\
 \mm{1}{1}{k} &= -\bigl((k+5)m_{k+4}-2(k+3)m_{k+2}+(k+1)m_k\bigr) & i=1,j=1\\
 \mm{0}{2}{k} &= (k+4)(k+3)m_{k+2}-2(k+2)(k+1)m_k+k(k-1)m_{k-2} & i=0,j=2\\
 \mm{2}{2}{k} &= (k+6)(k+5)m_{k+4}-2(k+4)(k+3)m_{k+2}+(k+2)(k+1)m_k & i=2,j=2
\end{align*}

The first \f{n-5} rows are identically zero because the maximum moment of $L_n$ involved is \f{m_{n-1}}. We are looking for a solution of the form \f{\vec{b}=\begin{bmatrix}\lambda_n & -2 & 1 & -1\end{bmatrix}^T}. The \f{(n-4)}-th row is \[\begin{bmatrix}m_n & -(n+1)m_n & 0 & (n+2)(n+1)m_n\end{bmatrix}\] and that is just enough to reconstruct the operator: \f{m_n(\lambda_n+2(n+1)-(n+2)(n+1))=0} and thus \f{\lambda_n = n(n+1)}, as expected. The reason for the original estimate \f{n+5} not being sharp is that some of the coefficients of $\Op$ (in fact, all but one) were known a priori.
\qed
\end{example}

Similar argument may be used in order to estimate $\nmC$. Straightforward computation leads to
\newcommand{\theF}{\mathfrak{G}_{\vec{\boldsymbol \alpha}}(x)}
\begin{prop}\label{prop:sizeC}
If \f{\vec{\boldsymbol \alpha}=(\alpha_{in})} is a solution of \eqref{eq:particular-solution}, then
\begin{align}\label{eq:particular-solution-moment-vanishing}
m_k \bigl( \theF  \bigr) = 0, \quad k=0,1,\dotsc,\nmC
\end{align}
where
\begin{align}\label{eq:difference-funct}
\theF &\isdef \fn(x)-\sum_{n=0}^\np \sum_{i=1}^\theorder \alpha_{in} \bas_i(x)
\qed
\end{align}
\end{prop}

Every function \f{\theF}  is a (piecewise) solution of \f{\Op f=0}. By Theorem \ref{thm:piecewise-recurrence} the moments of $\theF$ satisfy a linear recurrence relation \eqref{eq:piecewise-moments-recurrence-2}. Therefore, the maximal number of moments of \f{\theF \not\equiv 0} which are allowed to vanish is explicitly bounded by \eqref{eq:piecewise-moments-recurrence-2}. For instance, $\nmC$ may be taken as the length of the recurrence plus the value of the largest positive integer zero of its leading coefficient.

\begin{example}[Piecewise-constant functions on \f{[0,1]}]\label{ex:piecewise-const} Each piece of $\fn$ is a constant \f{\fn_i(x) \equiv c_i}. The operator \f{\Op=\prod_{i=1}^{\np}(x-\xi_i)\partial} annihilates entire $\fn$ and \f{\diffop(\shift)=\shift(\shift-\id)}.

\begin{itemize}
\item First we determine the minimal size of the system \eqref{eq:systemH}. For every polynomial $q(x)$ of degree $\np$, the moments of the function \f{f_{q}=\left(\diffop(x)q(x)\partial\right)\fn} are
\[
m_k(f_q)=\sum_{i=1}^\np q_i \xi_i^k, \quad q_i=\xi_i(\xi_i-1)q(\xi_i)
\]
Assume that the $\xi_i$'s are pairwise distinct and do not coincide with the endpoints. We claim that \f{\nmH=\np-1} is sufficient. Indeed, assume that \f{m_k(f_q)=0} for \f{k=0,1,\dots,\np-1}. Then
\[
X\vec{q}=\vec{0}: \qquad
X=\begin{bmatrix}
1 & 1 & \dots & 1\\
\xi_1 & \xi_2 & \dots &\xi_\np\\
\xi_1^2 & \xi_2^2 & \dots &\xi_\np^2\\
\vdots & \vdots &\ \vdots & \vdots\\
\xi_1^{\np-1} & \xi_2^{\np-1} & \dots &\xi_\np^{\np-1}
\end{bmatrix}, \;
\vec{q}= \begin{bmatrix}q_1\\q_2\\\vdots\\q_\np \end{bmatrix}
\]
The matrix $X$ is a Vandermonde matrix and it is nonsingular because \f{\xi_i \neq \xi_j} for \f{i \neq j}. It follows that \f{\vec{q}=\vec{0}} and therefore \f{q(\xi_i)=0} for all $i$. Therefore every solution of \eqref{eq:systemH} is a multiple of \f{\prod_{i=1}^\np (x-\xi_i)}.
\item Now we determine the minimal size of \eqref{eq:particular-solution}. The space $\nullsp{\Op}$ is spanned by piecewise-constant functions with the jump points $\xi_i$. Let \f{\vec{\boldsymbol \alpha}} be a solution to \eqref{eq:particular-solution} and let the function $\theF$ be as in \eqref{eq:difference-funct}. This $\theF$ is again a piecewise-constant function with $\np$ jump points $\xi_i$. Therefore, the moments of \f{\theF} satisfy the recurrence relation \eqref{eq:piecewise-moments-recurrence-2}. This recurrence has nonzero leading coefficient and its length is \f{\np+2}. Therefore, vanishing of the first \f{\np+1} moments of $\theF$ implies \f{\theF \equiv 0} and consequently \f{\nmC=\np} will be sufficient. The constants $c_i$ are precisely the solution of \eqref{eq:particular-solution}.
\qed
\end{itemize}
\end{example}

\section{Stability of inversion in a noisy setting}\label{sec:stability}
The presented inversion scheme assumes that the unknown signal \f{\fn \in \rpaclass} is clean from noise and that the moment sequence \f{\{m_k\}} is computed with infinite precision. In many applications this assumption is unrealistic. Therefore, it is practically important to analyze the sensitivity of the inversion to the noise in the data, both theoretically and numerically. In Section \ref{sec:stability-theoretical} below, we give a theoretical stability estimate for one special case of linear combinations of Dirac $\delta$-functions. Then we present results of numerical simulations for several test cases (Section \ref{sec:numerical-results}).

\newcommand{\nnn}[1]{\widetilde{#1}}
\newcommand{\nfn}{\nnn{\fn}}

\subsection{Theoretical stability analysis}\label{sec:stability-theoretical}
The general case of an arbitrary piecewise $D$-finite function appears to be difficult to analyze directly. Here we provide stability estimates for the reconstruction of the model
\begin{equation}\label{eq:delta-fun}
\fn(x) = \sum_{i=1}^{\np} a_i \delta(x-\xi_i)
\end{equation}

We argue that it is crucial to understand the behaviour of the reconstruction in this special case. Consider a generic \f{\fn \in \rpaclass}. Then \f{\Op \fn} is a linear combination of $\delta$-functions and their derivatives (see the proof of Theorem \ref{thm:piecewise-annihilating-operator}). Thus the model \eqref{eq:delta-fun} may be considered a ``prototype'' which captures the discontinuous nature of $\fn$. 

One can prove the following
\begin{thm}\label{thm:prony-stability}
Let $\fn$ be given by \eqref{eq:delta-fun}. Assume that the moments \f{m_k(\fn)} are known with accuracy $\er$. There exists a constant $C_1$ such that the parameters may in principle be recovered with the following accuracy:
\begin{align*}
 |\Delta \xi_i| &\leq C_1 \er a_i^{-1}\\
|\Delta a_i| &\leq C_1 \er
\end{align*}
where $C_1$ depends only on the geometry of \f{\xi_1,\dotsc,\xi_{\np}}:
\[
 C_1 \sim \prod_{i \neq j} |\xi_i-\xi_j|^{-1}
\]
\begin{proof}[Outline of the proof:]
 Write the Jacobian matrix of $\fwm$ explicitly and use the inverse function theorem to get the Jacobian of $\ivm$. This matrix may be factorized as a product of a diagonal matrix \f{\diag \{a_1,\dotsc,a_{\np}\}} and a Vandermonde matrix $V$ on the grid \f{\xi_1,\dotsc,\xi_{\np}}. Then $C_1$ may be taken to be the norm of \f{V^{-1}}.
\end{proof}
\end{thm}

The above result can be generalized to the case of the model
\[
 \fn(x) = \sum_{i=1}^{\np} \sum_{j=0}^{\polymaxorder_j} a_{ij} \der{\delta}{j}(x-\xi_i)
\]
We plan to present these results in detail separately.

The estimate of Theorem \ref{thm:prony-stability} reflects only the \emph{problem conditioning} (i.e. the sensitivity of the map $\ivm$ defined in Section \ref{sec:inversion} with respect to perturbations). Moreover, this estimate is sharp in the worst-case scenario, since it is based on directly evaluating the norm of the Jacobian.

On the other hand, an accurate analysis of the \emph{algorithm itself} should involve estimates for the condition numbers of the matrices $H$ and $C$ (see \eqref{eq:systemH} and \eqref{eq:particular-solution}, respectively) as well as the sensitivity of the root finding step (Figure \ref{fig:reconstruction-algorithm}). Let us briefly discuss these.
\begin{enumerate}
\item By Theorem \ref{thm:operator-coeffs-main-thm} and Remark \ref{rem:hankel-striped}, solving \eqref{eq:systemH} is equivalent to calculating a Hermite-Pad\'{e} approximant to \f{\Phi=\{1,h_0(z),\dotsc,h_{\theorder}(z)\}}. Estimates for the condition number of $H$ in terms of $\Phi$ are available (\cite{cabay1996cnp}). These estimates may hopefully be understood in terms of the differential operator $\Op$ and the function $\fn$ itself, using the connection provided by Theorem \ref{thm:operator-coeffs-main-thm}.
\item It is known (e.g. \cite{stewart1980bmi}) that if the coefficients of a (monic) polynomial are known up to precision $\er$, then the error in recovering a root of multiplicity $m$ may be as large as \f{\er^{\frac{1}{m}}}.
\item There is some freedom in the choice of the basis \f{\{\bas_i\}_{i=1}^\theorder} for the null space of $\Op$ (Proposition \ref{prop:particular-solution}). It should be investigated how this choice affects the condition number of $C$. 
\end{enumerate}
Considering the above, we hope that a stable reconstruction is possible at least for signals with relatively few and sufficiently separated points of discontinuity, as well as some mild conditions on the other model parameters.

\subsection{Numerical results}\label{sec:numerical-results}
In order to further justify the hope expressed in the previous section, we have performed several numerical simulations using a straightforward implementation  of the reconstruction algorithm from Section \ref{sec:inversion} (the implementation details are provided in Appendix \ref{apx:experiments}).

We have chosen 3 different models for the unknown signal: a rational function (no discontinuities),
a piecewise sinusoid, and a piecewise-polynomial function. In every simulation, the signal was first sampled on a sufficiently dense grid, then a white Gaussian noise of specified Signal-to-Noise Ratio (SNR) was added to the samples, and then the moments of the noised signal were calculated by trapezoidal numerical integration. To measure the success of the reconstruction, we calculated both the Mean Squared Error (MSE) as well as the error in the relevant model parameters. The results are presented in Figures \ref{fig:rational-rec}, \ref{fig:pwsin-rec} and \ref{fig:ppoly-rec}. Note that the SNR is measured in decibels (dB).

Several test signals were successfully recovered after a moderate amount of noise was added. Furthermore, increase in SNR generally leads to improvement in accuracy. Nevertheless, some  degree of instability is present, which is evident from the peaks in Figures \ref{subfig:rat-snr}, \ref{subfig:sin-snr}, \ref{subfig:ppoly-snr}. 

 While for some models, the increase in complexity leads to severe performance degradation, for other models it is not the case. Compare the growth in the degree of the rational function on one hand (Figure \ref{subfig:rat-deg}), and the increase in the number of jump points for the piecewise-constant reconstruction on the other hand (Figure \ref{subfig:ppoly-jumps}).

\begin{figure}
\subfigure[]{
\includegraphics[scale=0.45]{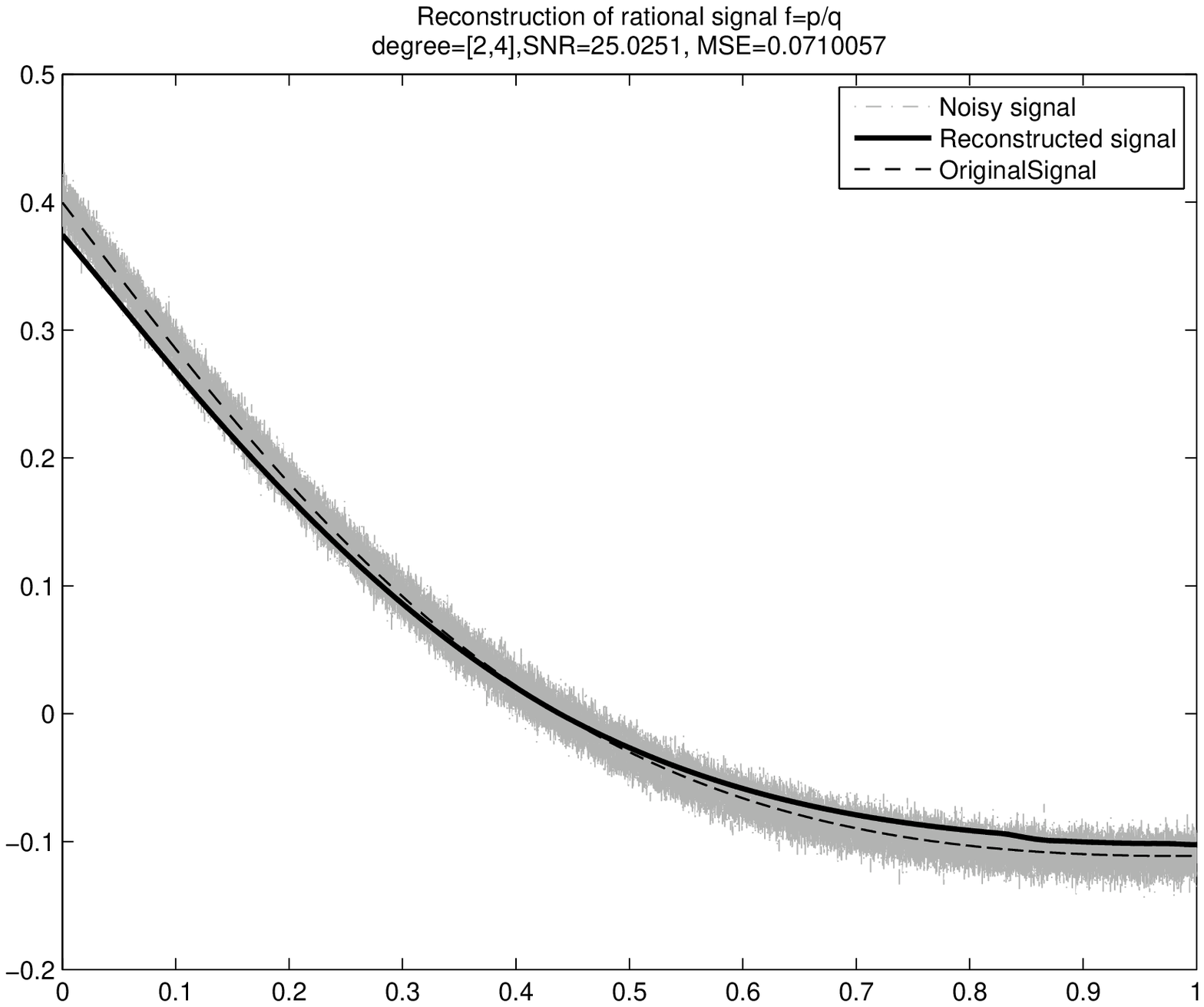}
\label{subfig:rat-test1}
}
\subfigure[]{
\includegraphics[scale=0.45]{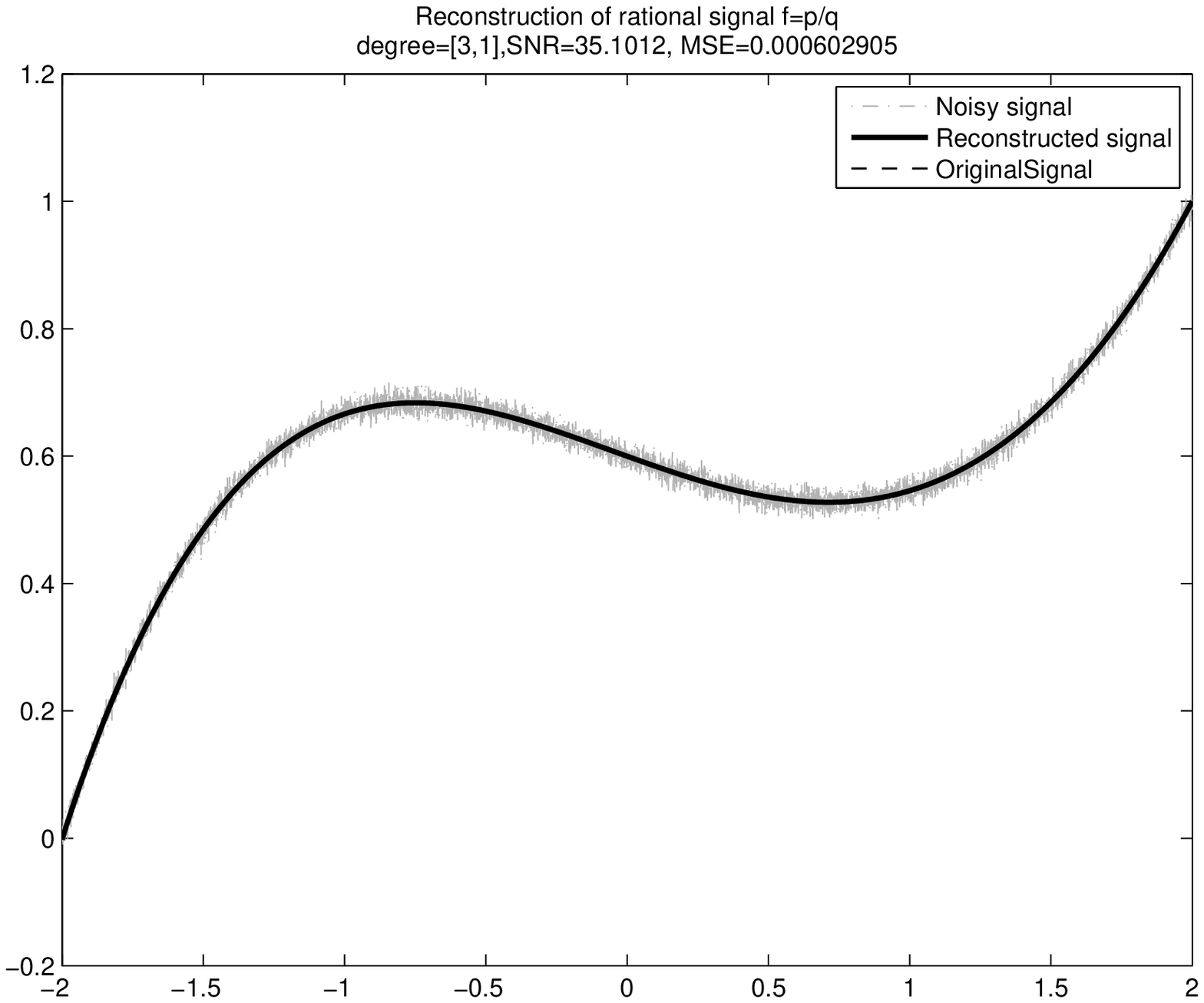}
\label{subfig:rat-test2}
}
\subfigure[]{
\includegraphics[scale=0.45]{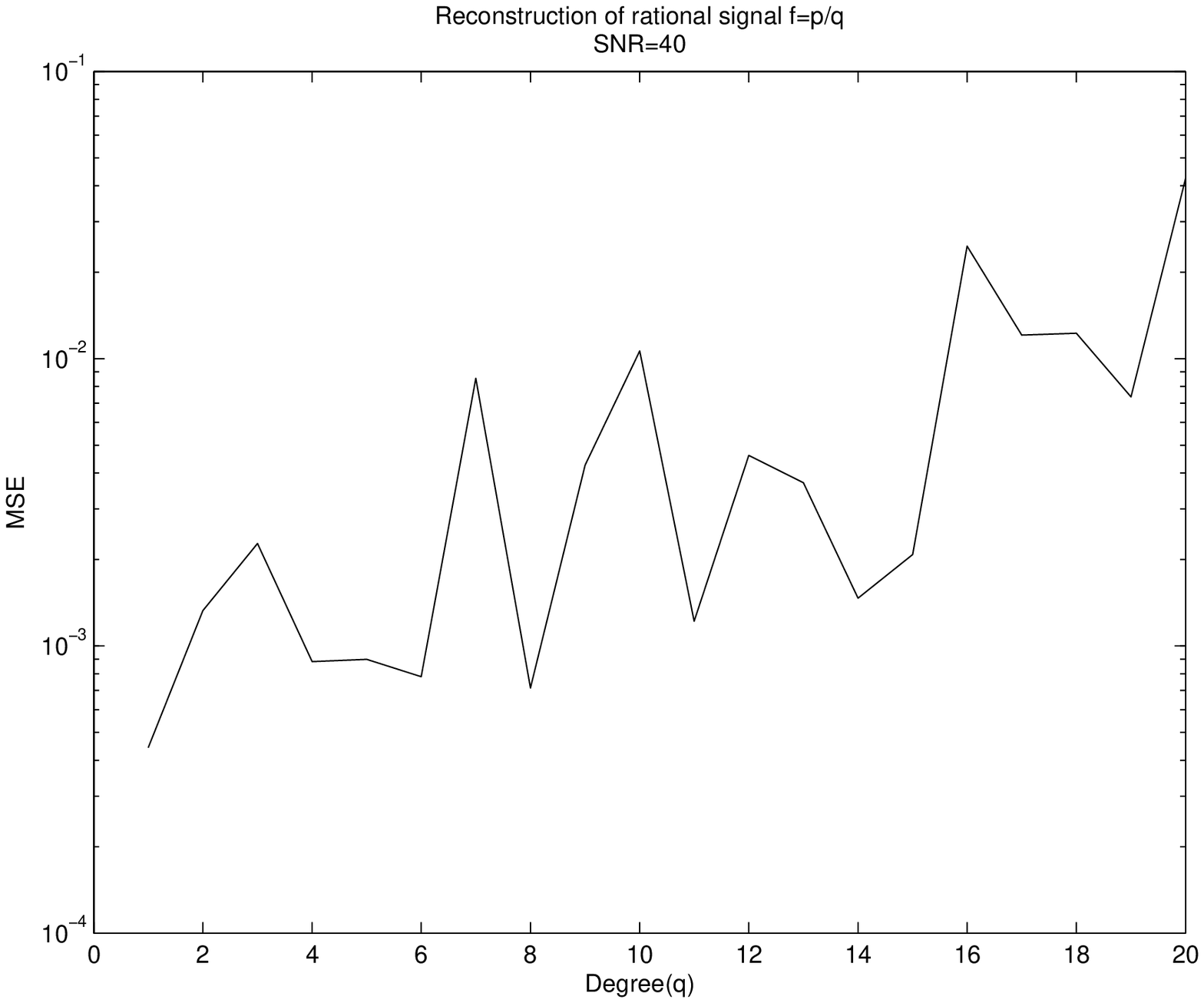}
\label{subfig:rat-deg}
}
\subfigure[]{
\includegraphics[scale=0.45]{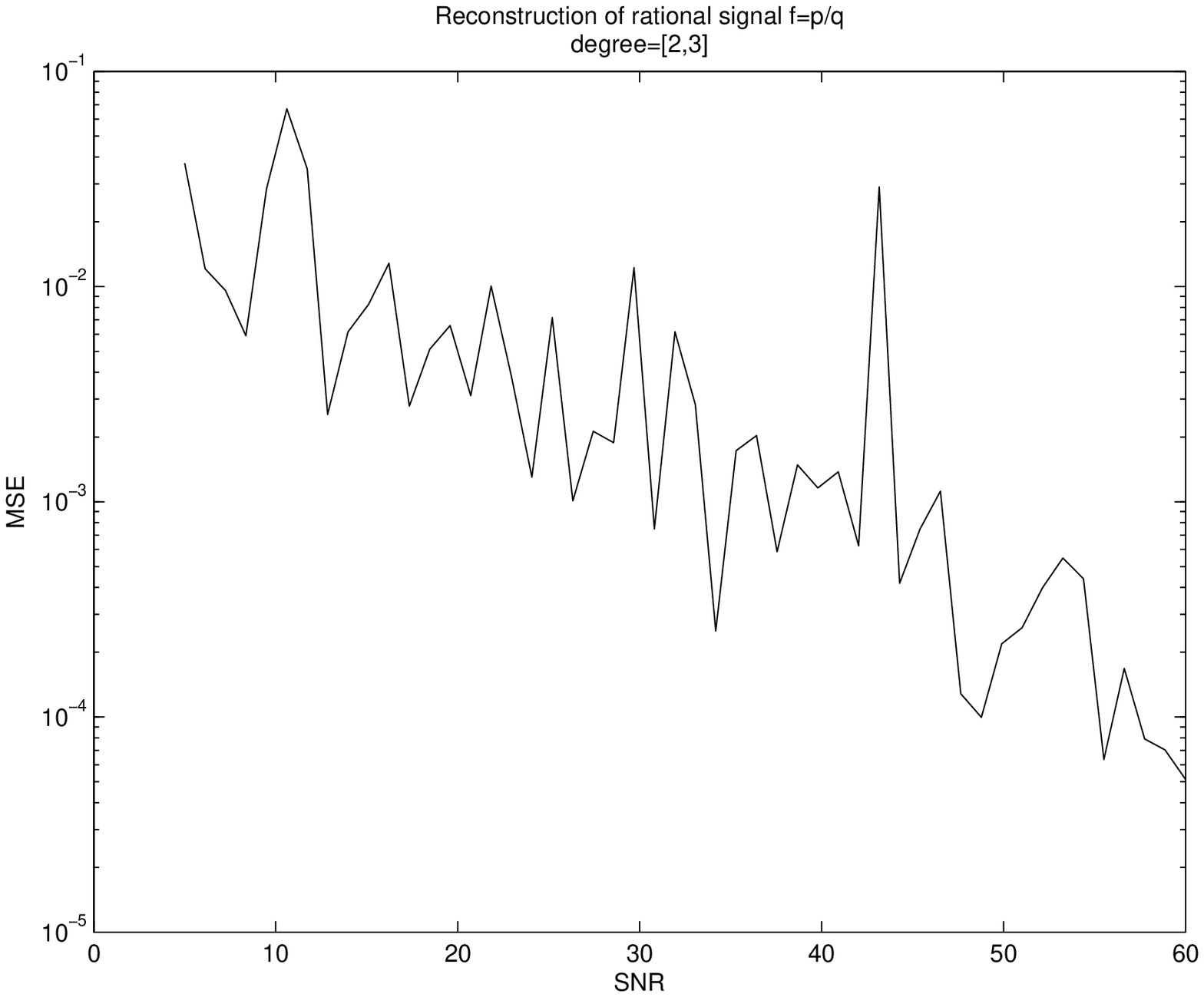}
\label{subfig:rat-snr}
}
\caption{{\bf Reconstruction of rational signals \f{f=p/q}.} \small \subref{subfig:rat-test1} The signal with \f{\deg p=2, \deg q=4} corrupted by noise with SNR=25 dB. Reconstruction MSE is 0.07. \subref{subfig:rat-test2} The signal with \f{\deg p=3, \deg q=1} corrupted by noise with SNR=35 dB. Reconstruction MSE is 0.0006. The reconstructed signal is visually indistinguishable from the original. \subref{subfig:rat-deg} Dependence of the MSE on the degree of the denominator with SNR=40 dB. \subref{subfig:rat-snr} Dependence of the MSE on the SNR with  \f{\deg p=2,\deg q=3}.}
\label{fig:rational-rec}
\end{figure}

\begin{figure}
\subfigure[]{
\includegraphics[scale=0.45]{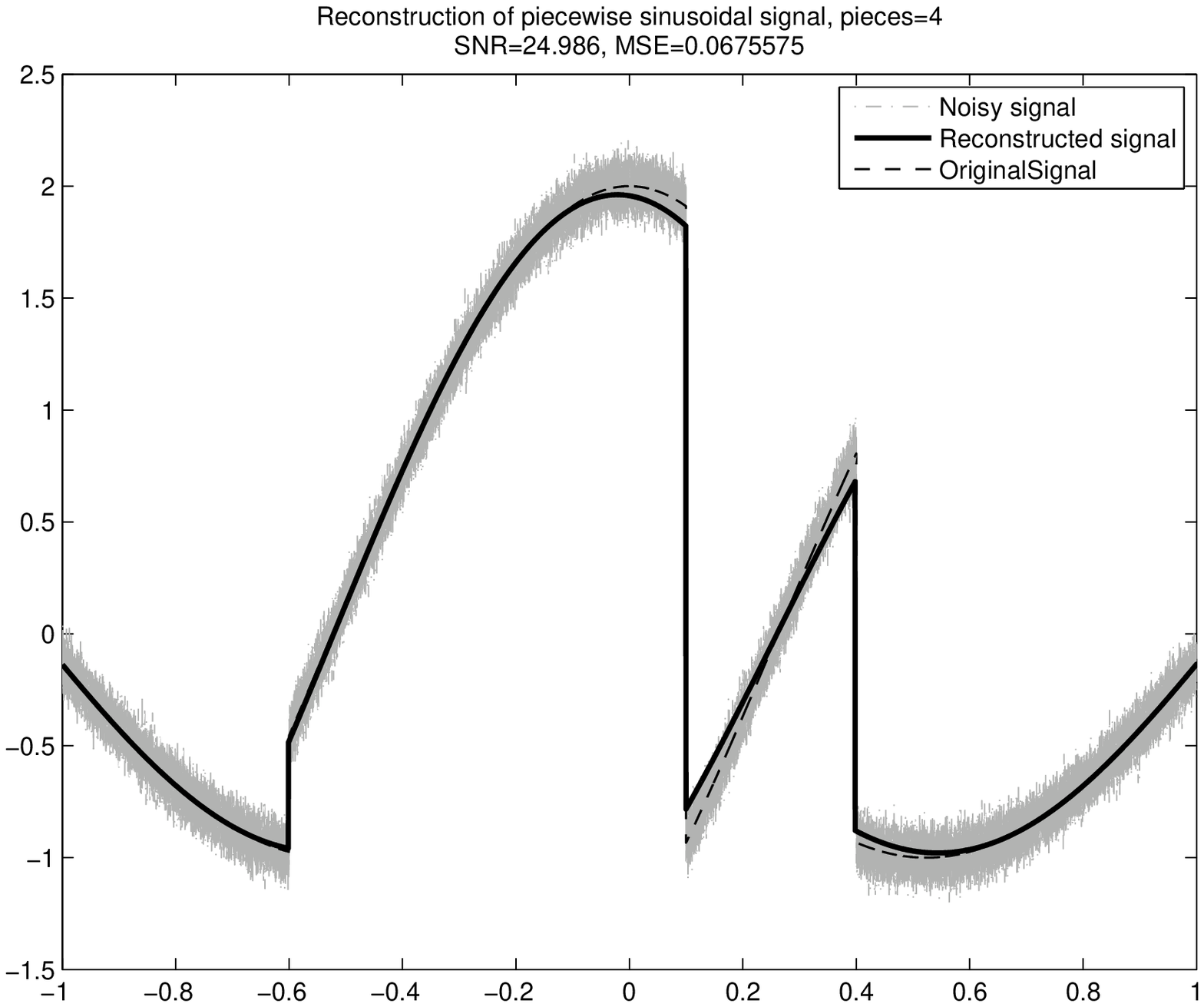}
\label{subfig:sin-test}
}
\subfigure[]{
\includegraphics[scale=0.45]{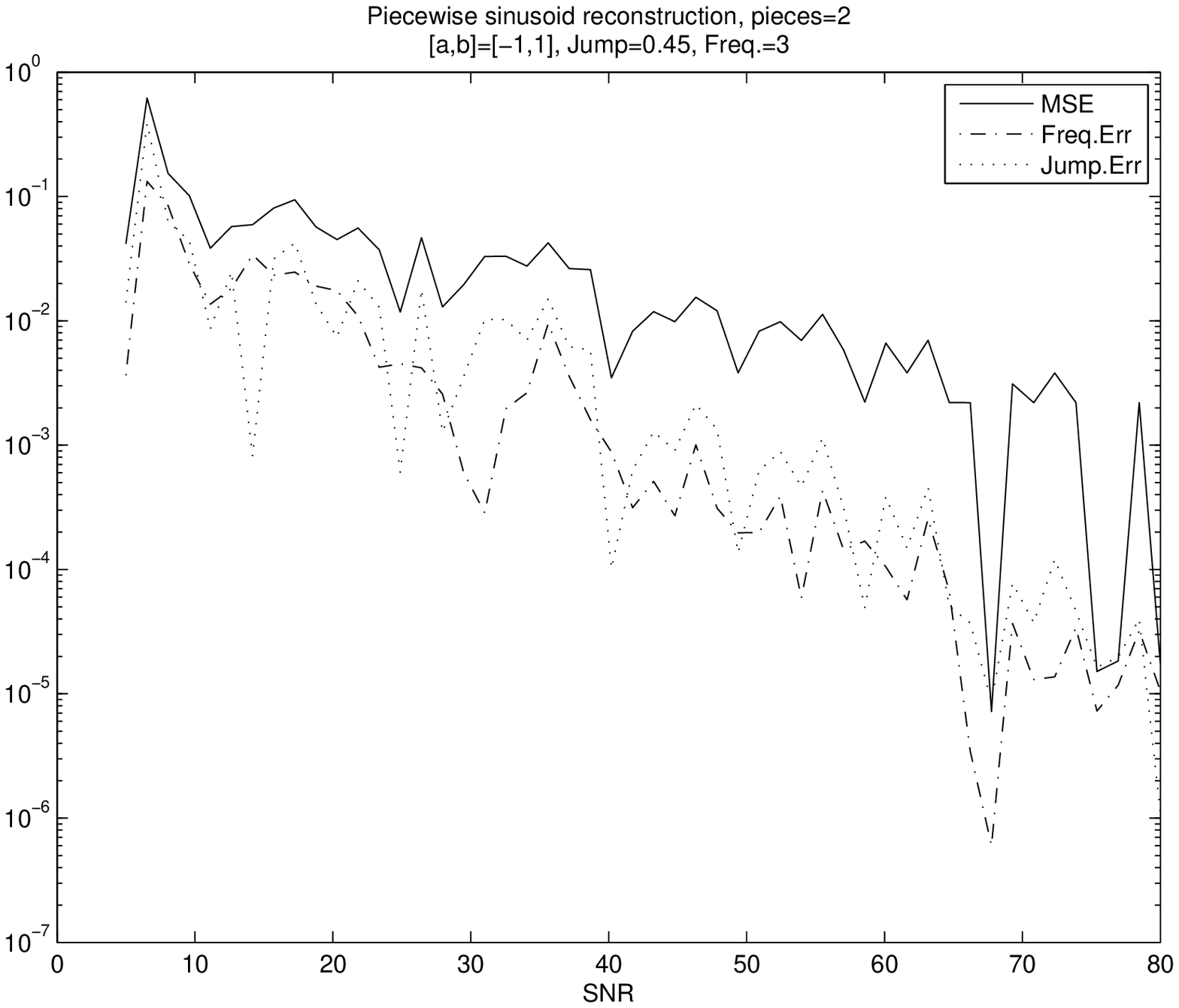}
\label{subfig:sin-snr}
}
\caption{{\bf Reconstruction of piecewise sinusoids.} \small \subref{subfig:sin-test} A sinusoid consisting of 4 pieces, corrupted by noise with SNR=25 dB. Reconstruction MSE=0.068. \subref{subfig:sin-snr} Dependence of the MSE on the SNR in the recontruction of a 2-piece sinusoid. Also plotted are the relative errors in the estimated location of the jump point and the estimated frequency.}
\label{fig:pwsin-rec}
\end{figure}

\begin{figure}
\subfigure[]{
\includegraphics[scale=0.45]{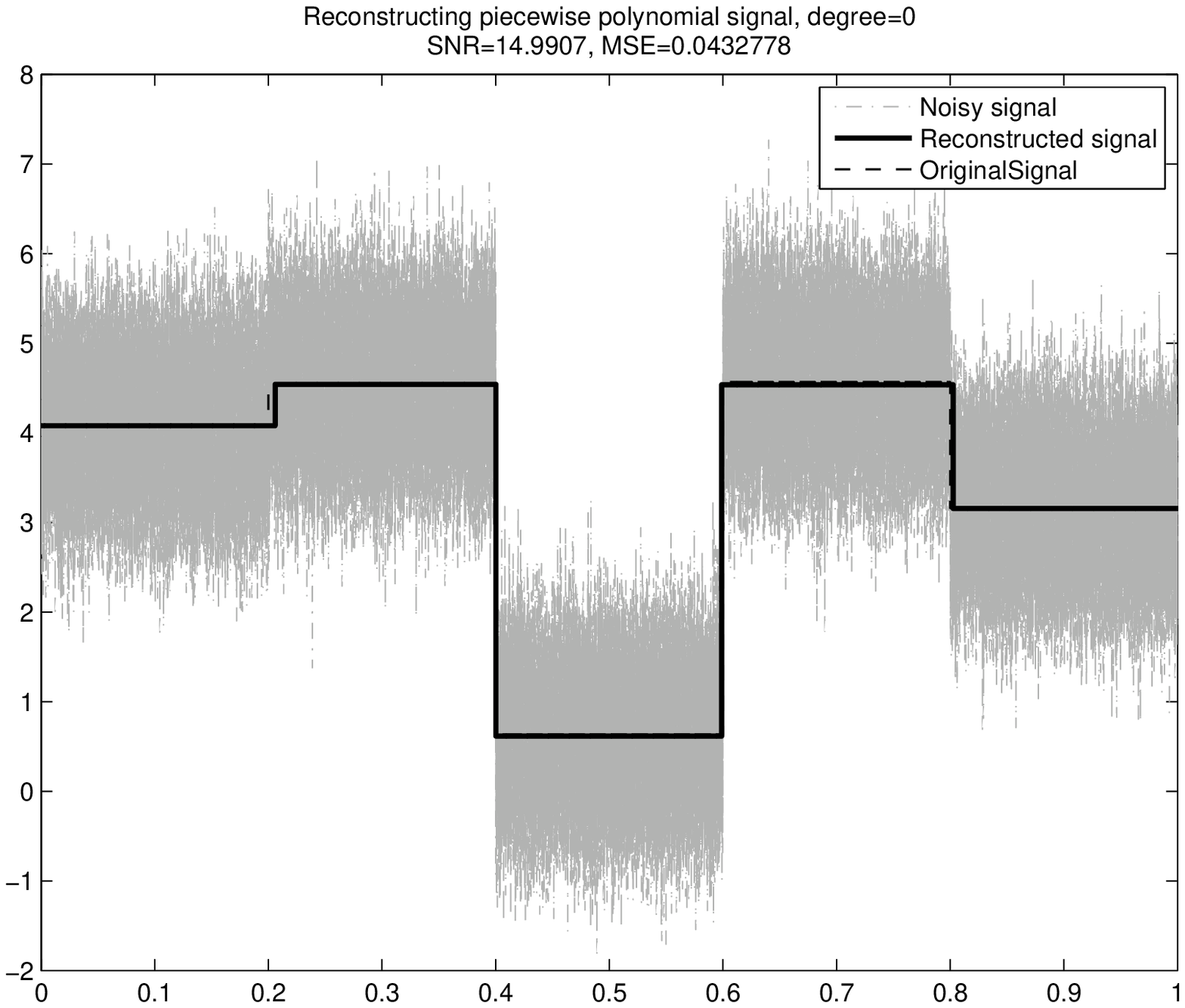}
\label{subfig:ppoly-test1}
}
\subfigure[]{
\includegraphics[scale=0.45]{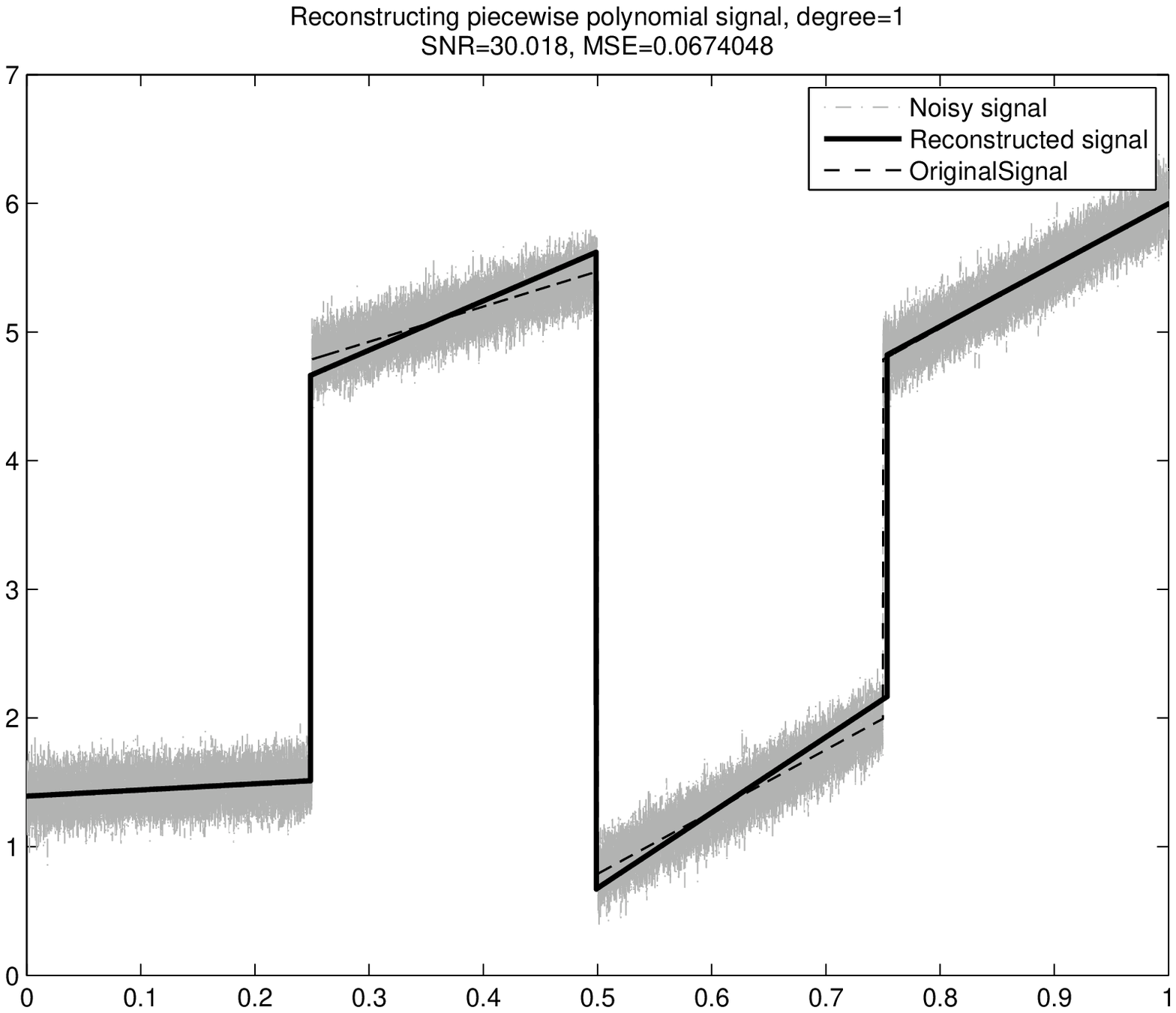}
\label{subfig:ppoly-test2}
}
\subfigure[]{
\includegraphics[scale=0.45]{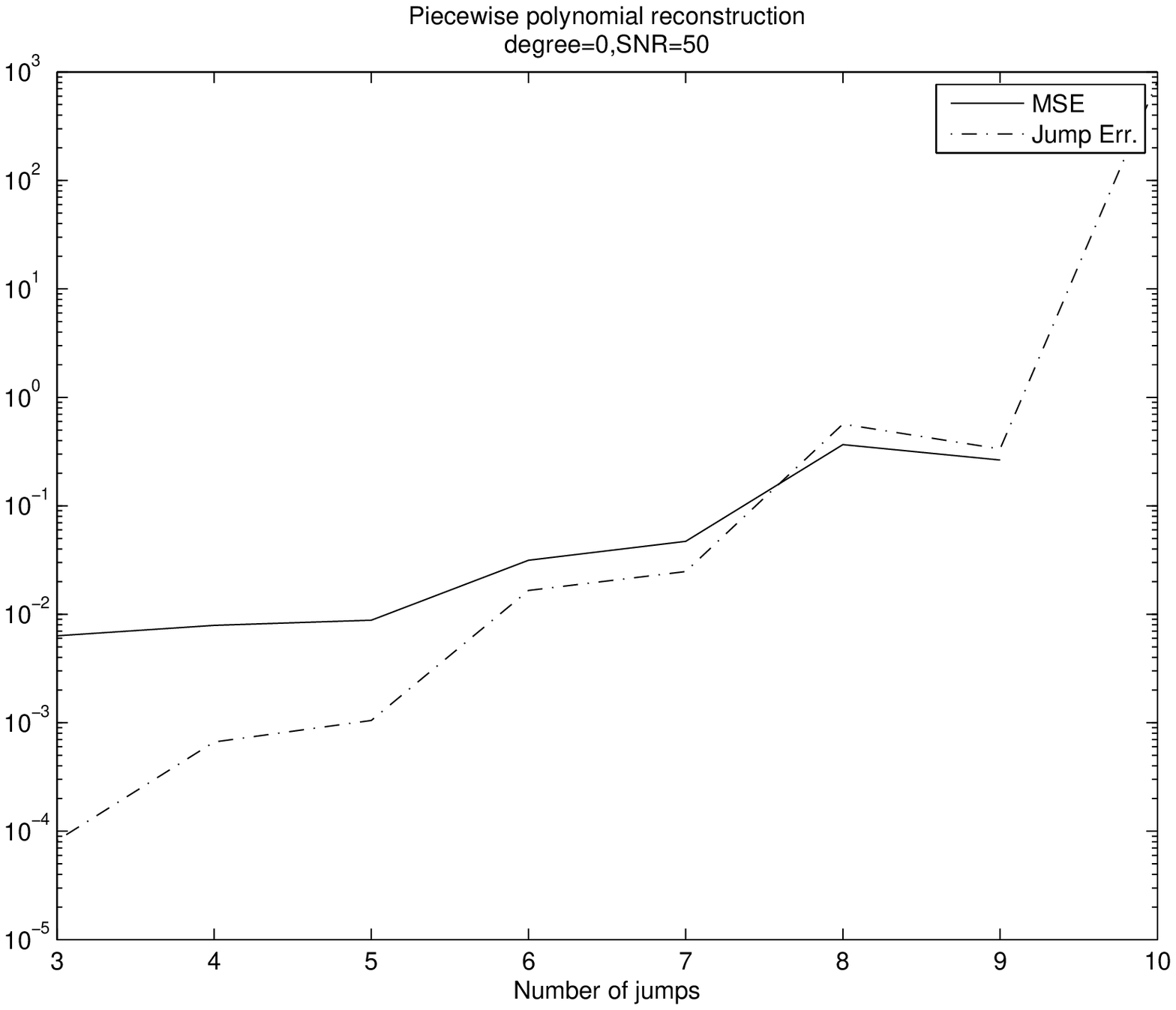}
\label{subfig:ppoly-jumps}
}
\subfigure[]{
\includegraphics[scale=0.45]{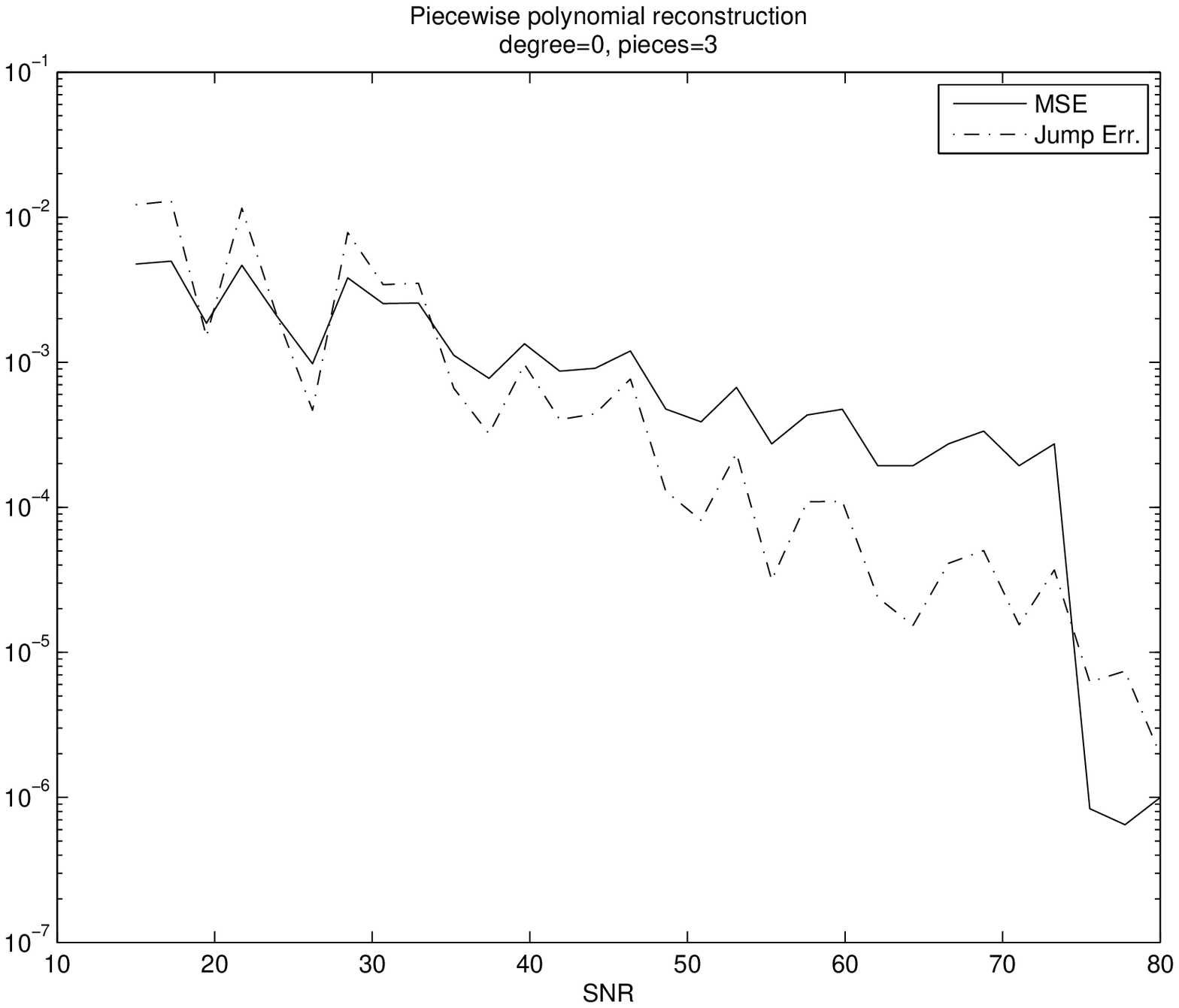}
\label{subfig:ppoly-snr}
}
\caption{{\bf Reconstruction of piecewise-polynomials.} \small \subref{subfig:ppoly-test1} Piecewise-constant signal with 5 jumps corrupted by noise with SNR=15 dB. Reconstruction MSE is 0.043. The reconstructed signal is visually indistinguishable from the original. \subref{subfig:ppoly-test2} Piecewise-linear signal with 3 jumps, corrupted by noise with SNR=30 dB. Reconstruction MSE is 0.067. \subref{subfig:ppoly-jumps} Dependence of the MSE on the number of jumps (\f{\np}) for a piecewise-constant signal, with SNR=50 dB.  The reconstruction failed for \f{\np=10}. \subref{subfig:ppoly-snr} Dependence of the MSE on the SNR for a piecewise constant signal with 2 jumps. \subref{subfig:ppoly-jumps}+\subref{subfig:ppoly-snr} Also plotted is the relative error in the recovered jump location.}
\label{fig:ppoly-rec}
\end{figure}

\section{Discussion}\label{sec:discussion}

The piecewise $D$-finite moment inversion problem appears to be far from completely solved. In the theoretical direction, the solvability conditions need to be further refined. In particular, we hope that minimality results (``what is the minimal number of measurements which is sufficient to recover the model uniquely?'') in the spirit of Examples \ref{ex:orthopoly} and \ref{ex:piecewise-const} concerning the reconstruction of additional classes of functions may be obtained using the methods presented in this paper (Section \ref{sec:solvability}). The importance of this question is discussed e.g. in \cite{sig_ack}, where estimates on the finite moment determinacy of piecewise-algebraic functions are given. In this context, the role of the moment generating function is not yet fully understood. Another open question is the analysis of the case in which the function satisfies different operators on every continuity interval.

In the numerical direction, the results of Section \ref{sec:stability} suggest that a ``naive'' implementation of the presented algorithm is relatively accurate for simple enough signals corrupted by low noise levels. We believe that attempts to improve the robustness of the algorithm should proceed in at least the following directions:
\begin{enumerate}
\item A similar question regarding stability of reconstruction of signals with finite rate of innovation is addressed in \cite{maravic2005sar}. We propose to investigate the applicability of that method to our reconstruction problem. 
\item The connection of the system \eqref{eq:systemH} to Hermite-Pad\'{e} approximation may hopefully be exploited to build some kind of sequence of approximants which converge to the true solution, as more elements of the moment sequence are known. In fact, a similar approach (with standard Pad\'{e} approximants) has been used in the ``noisy trigonometric moment problem'' (\cite{marchbarone:2000}).
\end{enumerate}

Furthermore,  we believe that it is important to understand how the various model parameters influence the stability of the algorithm. The most general answer in our context would be to give stability estimates in terms of the differential operator $\Op$ and the geometry of the jump points.

On the other hand, the algebraic moments \eqref{eq:usual-moments} are known to be a non-optimal choice for measurements (see \cite{talenti1987rff}) due to their strong non-orthogonality. As  stated in the Introduction, the moment inversion is a ``prototype'' for some real-world problems such as reconstruction from Fourier measurements. In fact, it is known that the Fourier coefficients (as well as coefficients with respect to other orthogonal systems) of many ``finite-parametric'' signals  satisfy various relations, and this fact has been utilized in numerous schemes for nonlinear Fourier inversion (\cite{berent2006prs,dragotti2007sma,eckhoff1995arf,kvernadze2004ajd,banerjee1997eau,beckermann2008rgp}). The point of view presented in this paper, namely, the differential operator approach, can hopefully be generalized to include these types of measurements as well. In this regard, we expect that methods of holonomic combinatorics (\cite{zeilberger1990hsa}) may provide useful insights.

\bibliographystyle{plain}
\bibliography{../../bibliography/all-bib}

\appendix
\section{The inversion algorithm - implementation details}\label{apx:experiments}
We have implemented the inversion algorithm of Section \ref{sec:inversion} in the MATLAB environment (\cite{Sigmon:1998:MP}). For each of the three types of signals we built a specific inversion routine. All the three routines have a common base as follows:
\begin{itemize}
\item The matrix $H$ is chosen to be square. Solution to the system \f{H\vec{a}=\vec{0}} is obtained by taking the highest coefficient of $\vec{a}$ to be 1 and then performing standard Gaussian elimination with partial pivoting on the reduced system.
\item The root finding step is performed with the \verb!pejroot! routine of the \verb!MULTROOT! package (\cite{zeng2004amm,zeng2005cmr}), which takes into account the multiplicity structure of the polynomial.
\item The calculation of the moments in the matrix $C$ is done via numerical quadrature. 
\end{itemize}

The difference between the routines lies in the construction of $H$ and the computation of the basis for the nullspace of $\Op$. These are determined by the signal type, as described below.
\begin{enumerate}
\item{\underline{Rational functions:}}
 A rational function \f{f(x)=\frac{p(x)}{q(x)}} is annihilated by the first order operator
 \[
  \Op = (-pq)\partial + (p'q-pq')\id
 \]
Therefore the matrix $H$ consists of two ``stripes'' $V_0$ and $V_1$ (the structure of the matrices $V_i$ is given by \eqref{eq:stripe-def}). If \f{\deg p(x)=r} and \f{\deg q(x)=s}, then the degrees of the coefficients of \f{\Op} are \f{\polymaxorder_0=r+s-1} and \f{\polymaxorder_1=r+s}. Subsequently, \f{V_0} is \f{(\polymaxorder_0+\polymaxorder_1+2) \times (\polymaxorder_0+1)} and \f{V_1} is \f{(\polymaxorder_0+\polymaxorder_1+2) \times (\polymaxorder_1+1)}. After \f{\Op} is reconstructed (without explicitly recovering $p$ and $q$), a solution $\bas$ to \f{\Op \bas \equiv 0} is obtained by numerically solving (using the \verb!ode45! solver) the initial value problem \f{\{\Op \bas(x) \equiv 0, \bas(a)=1,x=a \dots b\}}. Then the final solution is obtained as
\[
 \widetilde{f}(x) = \frac{\int_a^b \bas(x)dx}{m_0} \bas(x)
\]

\item{\underline{Piecewise-sinusoids:}} Let \f{\fn(x)} consist of \f{\np+1} pieces of the form
\[
  \fn_n(x) = A_n \sin(\omega x+\varphi_n)
\]
The annihilating operator for each piece is \f{\Op=\partial^2+\omega^2 \id}, therefore the entire $\fn$ is annihilated by
\[
 \newop = \bigl\{(x-\xi_1)^2 \cdot \dotsc \cdot (x-\xi_{\np})^2\bigr\}(\partial^2+\omega^2\id)
\]
So \f{H=\begin{bmatrix}V_0 & V_2\end{bmatrix}}, where both $V_0$ and $V_2$ are \f{2(2\np+1) \times (2\np+1)}. From  the operator reconstruction step we obtain \f{\recOp = \coeff_1(x) \partial^2 + \coeff_0(x)\id}. The jump points are taken to be the arithmetic means of the corresponding roots of \f{\coeff_0} and \f{\coeff_1}. Because of the normalization performed in the operator recovery step, the frequency \f{\omega'} may be taken to be the square root of the highest coefficient of $\coeff_0$. The reconstruction is considered to be unsuccessful if at least one root fails to lie inside the interval \f{[a,b]}.

The basis for the nullspace of \f{\Op} is chosen to be \f{\bas_1=\sin(\omega' x)} and \f{\bas_2=\cos (\omega' x)}.

\item{\underline{Piecewise-polynomials:}} Let \f{\fn} consist of \f{\np+1} polynomial pieces \f{\fn_n}, where \f{\deg \fn_n = d}. The annihilating operator for every piece is \f{\Op=\partial^{d+1}}, and therefore $\fn$ is annihilated by
\[
 \newop = \bigl\{(x-\xi_1)^{d+1} \cdot \dotsc \cdot (x-\xi_{\np})^{d+1} \bigr\} \partial^{d+1}
\]
Consequently, \f{H=V_{d+1}} is a \f{\np(d+1) \times \np(d+1)} Hankel matrix. Operator reconstruction step gives \f{\recOp=p(x) \partial^{d+1}}, and the jump locations are the roots of $p(x)$, each with multiplicity \f{d+1}. The reconstruction is considered to be unsuccessful if at least one root fails to lie inside the interval \f{[a,b]}.

The basis for the nullspace of \f{\Op} is always chosen to be \f{\{1,x,\dotsc,x^d\}}.
\end{enumerate}

\end{document}